\documentclass[reqno]{amsart}
\usepackage{amsmath,amscd,amsthm,amsfonts,amsopn,amssymb,color,ulem}

\usepackage{verbatim}

\numberwithin{equation}{section}

\newtheorem{theorem}{Theorem}[section]

\newtheorem{lemma}[theorem]{Lemma}
\newtheorem{proposition}[theorem]{Proposition}

\theoremstyle{definition}

\theoremstyle{definition}
\newtheorem{definition}[theorem]{Definition}

\theoremstyle{definition}

\DeclareMathOperator{\divop}{div\,}
\DeclareMathOperator{\supp}{supp}
\DeclareMathOperator{\iRe}{Re\,}
\newcommand{\tth}{\tilde \theta}
\newcommand{\tu}{\tilde u}
\newcommand{\bR}{\mathbb{R}}

\newcommand{\bT}{\mathbb{T}}

\allowdisplaybreaks
\begin{document}

\title[Nonlinear instability for the surface quasi-geostrophic equation]{Nonlinear instability for the surface quasi-geostrophic equation in the supercritical regime}
\author[Aynur Bulut and Hongjie Dong]{Aynur Bulut and Hongjie Dong}
\address[A. Bulut]{Department of Mathematics, Louisiana State University, Baton Rouge, LA 70803-4918, USA}
\email{aynurbulut@lsu.edu}

\address[H. Dong]{Division of Applied Mathematics, Brown University,
182 George Street, Providence, RI 02912, USA}
\email{Hongjie\_Dong@brown.edu}

\thanks{Oct. 4, 2019}

\begin{abstract}
We consider the forced surface quasi-geostrophic equation with supercritical dissipation.  We show that linear instability for steady state solutions leads to their nonlinear instability.  When the dissipation is given by a fractional Laplacian, the nonlinear instability is expressed in terms of the scaling invariant norm, while we establish stronger instability claims in the setting of logarithmically supercritical dissipation.  A key tool in treating the logarithmically supercritical setting is a global well-posedness result for the forced equation, which we prove by adapting and extending recent work related to nonlinear maximum principles.  We believe that our proof of global well-posedness is of independent interest, to our knowledge giving the first large-data supercritical result with sharp regularity assumptions on the forcing term.
\end{abstract}

\maketitle

\section{Introduction and main results}

The surface quasi-geostrophic (SQG) equation is a mathematical model for a rapidly rotating fluid in certain asymptotic
regimes.  It has particular significance as a two-dimensional fluid model which captures several difficulties
arising in the study of the three-dimensional Euler and Navier-Stokes systems.  In this context, it is of particular
interest to understand instability phenomena for this and related evolution equations.

In \cite{FSV}, Friedlander, Strauss, and Vishik studied nonlinear instability phenomena for a class of abstract
evolution equations in Banach spaces.  They showed that linear instability of a steady state solution implies its
nonlinearly instability.  Here, linear instability is expressed in terms of a spectral condition, while
nonlinear instability is understood in the sense of Lyapunov (see Definition \ref{def-lyaponov} below).
Results of this type for ODEs are classical, and were earlier known for the Navier-Stokes
equations (due to Yudovich \cite{Yu}).  The class of abstract equations studied in \cite{FSV} includes
the two-dimensional Euler equation.

When the results of \cite{FSV} are applied to the critical SQG equation, they establish the linear
instability--nonlinear instability implication with respect to $H^s(\mathbb{T}^2)$ norms with $s>2$.
Noting this, in \cite{FPV}, Friedlander, Pavlovi\'c, and Vicol studied the question of nonlinear instability
with respect to the $L^2(\mathbb{T}^2)$ norm for this equation.  Their argument is based on semigroup
estimates for the linearized operator and a bootstrap technique.  To close the bootstrap estimates, they establish
a global Lipschitz bound for solutions to the forced equation (the forcing term is added to permit the existence of
linearly-unstable steady state solutions).

In this paper, we establish a related class of nonlinear instability properties for the supercritical surface
quasi-geostrophic equation.  As in \cite{FPV}, we include forcing terms, which lead to the existence of
linearly-unstable steady state solutions.  We obtain two sets of results.  The first corresponds to dissipation
given by the fractional Laplacian $\Lambda^\alpha$, $\Lambda=(-\Delta)^{1/2}$, in the supercritical range $0<\alpha<1$,
for which we establish a global perturbation estimate which allows us to use a bootstrap argument to control
higher-order norms.

In our second class of results, we establish nonlinear instability results measured in norms of lower
regularity, for which stronger global information about the evolution is required.  We treat the case of
logarithmically-supercritical dissipation, and use the fact that the equation is only slightly
supercritical to establish a robust global well-posedness theory for it.  In particular, we obtain
uniform-in-time $H^k$ control over solutions for $k\geq 0$.  This global result allows
us to prove a stronger nonlinear instability result: we show that linear instability of the steady state
implies its nonlinear instability with respect to the $L^2(\mathbb{T}^2)$ norm.

We now give the precise statement of our results, beginning with the SQG equation associated to the
supercritical fractional Laplacian.  Let $R^\perp$ denote the rotated vector-valued Riesz transform, i.e.
$R^\perp\theta=(R_2\theta,-R_1\theta)$, where $R_i=\partial_i(-\Delta)^{-1/2}$.  Fix $0<\gamma<1$, and
suppose that $f\in H^{2-\gamma}(\mathbb{T}^2)$ satisfies the mean-zero condition
\begin{align}
\int fdx=0.\label{eq-mean-zero}
\end{align}

We consider the forced SQG equation with supercritical dissipation of order $\gamma$:
\begin{align}
\partial_t\theta+R^\perp\theta\cdot\nabla\theta+\Lambda^\gamma\theta=f,\label{eq1-supercritical}
\end{align}
with the initial condition $\theta(0,\cdot)=\theta_0(\cdot)$.  Suppose that $\Theta_0$ solves the stationary
equation
\begin{align}
R^\perp \Theta_0\cdot \nabla \Theta_0+\Lambda^\gamma\Theta_0=f.\label{eq-steady-supercritical}
\end{align}
Letting $\tth=\tth(t,x)$ be an unknown perturbation and writing $\theta(t,x)=\Theta_0(x)+\tth(t,x)$, the
perturbed function $\theta$ solves (\ref{eq1-supercritical}) if and only if $\tth$ solves
\begin{align}
\partial_t\tth=L_\gamma\tth+N(\tth),\label{eq-perturbation-supercritical}
\end{align}
with
\begin{align}
L_\gamma\tth=-(R^\perp\Theta_0)\cdot \nabla\tth-(R^\perp\tth)\cdot\nabla\Theta_0-\Lambda^\gamma\tth\label{linearized-supercritical}
\end{align}
and
\begin{align}
N(\tth)=-(R^\perp\tth)\cdot\nabla\tth.\label{def-n}
\end{align}

\begin{definition}[Lyapunov stability]
\label{def-lyaponov}
Let $X$ and $Z$ be two Banach spaces.  We say that a solution $\Theta_0$ to (\ref{eq-steady-supercritical}) is
an $(X,Z)$-nonlinearly stable steady-state solution to (\ref{eq1-supercritical}) if for any $\varepsilon_0>0$ there
exists $\varepsilon_1=\varepsilon_1(\varepsilon_0)>0$ such that for all $\tth_0\in X$ satisfying
$\lVert \tth_0\rVert_{Z}<\varepsilon_1$ there exists a global solution $$\tth\in C([0,\infty);X)$$ to
(\ref{eq-perturbation-supercritical}) with $\tth(0)=\tth_0$ and
\begin{align*}
\sup_{t>0}\,\, \lVert \tth(t)\rVert_{Z}<\varepsilon_0.
\end{align*}
\end{definition}

We say that a solution $\Theta_0$ to (\ref{eq-steady-supercritical}) is a $(X,Z)$-nonlinearly unstable steady-state
solution to (\ref{eq1-supercritical}) if it is not $(X,Z)$-nonlinearly stable.  Our first main result shows that
instability of the linearized operator $L_\gamma$ based at $\Theta_0$ implies nonlinear instability of $\Theta_0$.

\begin{theorem}
\label{thm1}
Let $\gamma\in (0,1)$ be given and fix $f\in H^{3}(\mathbb{T}^2)$.  Suppose that $\Theta_0\in H^{2-\gamma}(\bT^2)$ is a solution to
the steady-state supercritical SQG equation (\ref{eq-steady-supercritical}) which is linearly unstable in the sense
that the linearized operator $L_\gamma$ defined by (\ref{linearized-supercritical}) has an eigenvalue $\lambda$
with $\iRe\lambda>0$.  Then $\Theta_0$ is $(H^{2-\gamma},H^{2-\gamma})$-nonlinearly unstable.
\end{theorem}

As we described above, to establish instability results below the critical local well-posedness regularity threshold,
we require stronger global information on the evolution.  Toward this end, we fix $0<a<1/2$, $\tau\in (1,5/3)$, and
consider the forced logarithmically supercritical SQG with forcing function
$f\in H^{\tau}(\mathbb{T}^2)$ satisfying the mean-zero condition (\ref{eq-mean-zero}):
\begin{align}
\partial_t\theta+R^\perp\theta\cdot\nabla\theta+\mathcal{L}\theta=f,\label{eq1-log}
\end{align}
where $\mathcal{L}$ is a logarithmically supercritical diffusion operator given as a Fourier multiplier by
\begin{align}
\widehat{\mathcal{L}\theta}(\xi)=\frac{|\xi|\widehat{\theta}(\xi)}{\log^a(\kappa+|\xi|)}\label{eq-L}
\end{align}
for $\xi\in\mathbb{R}^2$, where $\kappa\ge e$ is a constant.

In certain cases, global well-posedness results for the critical SQG can be extended to slightly supercritical
settings (see, e.g. \cite{DKSV,XZ}).  This is the motivation behind our interest in
the logarithmically supercritical dissipation $\mathcal{L}$.  In Section $3$, we refine and develop a method
introduced in \cite{CZV} to establish global well-posedness for the initial value problem associated to
(\ref{eq1-log}) with data in $H^\tau(\mathbb{T}^2)$ (see also \cite{CTV} for an application of related ideas
to the forced critical problem).  In particular, we prove the following theorem:

\begin{theorem}
\label{thm-log-gwp}
Let $0<a<1/2$, $0<\varepsilon<1$, and $1<\tau<5/3$ be given.  Suppose that $f\in H^\tau(\mathbb{T}^2)$ satisfies the mean-zero condition (\ref{eq-mean-zero}).  Then for all $\theta_0\in H^\tau(\mathbb{T}^2)$ there exists a global solution
$$
\theta\in C^0([0,\infty);H^\tau(\mathbb{T}^2))\cap L^2_{\textrm{loc}}([0,\infty);H^{\tau+1/2-\varepsilon}(\mathbb{T}^2))
$$
to (\ref{eq1-log}) with $\theta(0)=\theta_0$.  Moreover, for any $k\geq 0$ and $t_0>0$, $$\sup_{t>t_0}\,\, \lVert \theta(t)\rVert_{H^k}<\infty$$
provided that $f\in H^k(\mathbb{T}^2)$.
\end{theorem}

As we mentioned above, the proof of Theorem $\ref{thm-log-gwp}$ is based on the nonlinear lower bound method of
\cite{CC,CV,CTV,CZV}, adapted to the logarithmically supercritical setting.  Our arguments are most closely related
to the conditional supercritical result in \cite{CZV}, where the level of supercriticality depends on the size
of the data.  The method is based on pairing a quantitative version of ``eventual regularization'' results in
the spirit of \cite{Sil,Dab} with a suitable form of the local theory for existence of strong solutions.  In our
setting, the logarithmic supercriticality allows to close the gap between these two sides and establish a full
global wellposedness result.  To the best of our knowledge, this is the first case in which these methods are
applied to establish global results in the supercritical regime for arbitrarily large data,
and we believe that our techniques are of substantial independent interest.

With the statement of the global results in hand, we now state our nonlinear instability result for this equation.  Let $\Theta_0$ be a solution to the stationary equation
\begin{align}
R^\perp \Theta_0\cdot \nabla \Theta_0+\mathcal{L}\Theta_0=f.\label{eq-steady-log}
\end{align}
In analogy with (\ref{eq-perturbation-supercritical}), $\theta=\Theta_0+\tth$ solves (\ref{eq1-log}) if and only if the perturbation $\tth$ solves
\begin{align}
\partial_t\tth=L_{\textrm{log}}\tth+N(\tth),\label{eq-perturbation}
\end{align}
with
\begin{align}
L_{\textrm{log}}\tth=-(R^\perp\Theta_0)\cdot \nabla\tth-(R^\perp\tth)\cdot\nabla\Theta_0-\mathcal{L}\tth\label{linearized-log}
\end{align}
and where $N(\tth)$ is as defined in (\ref{def-n}).

\begin{theorem}
\label{thm2}
Fix $f\in H^3(\mathbb{T}^2)$.  Suppose that $\Theta_0\in H^\tau(\mathbb{T}^2)$ is a solution to the steady-state logarithmically supercritical SQG equation (\ref{eq-steady-log}) such that the linearized operator $L_{\textrm{log}}$ defined by (\ref{linearized-log}) has an eigenvalue $\lambda$ with $\iRe\lambda>0$.  Then $\Theta_0$ is $(H^{\tau},L^{2})$-nonlinearly unstable.
\end{theorem}

We conclude this introduction with a few comments on notation, and an outline of the rest of the paper.  In the rest of this paper, we use $A\lesssim B$ (or $A\gtrsim B$) to denote the condition that $A\leq C B$ (or $A\geq C B$) holds for some $C>0$.  The constants $C$ (and likewise $C_1$, $C_2$, etc.) may change from line to line unless otherwise indicated.  In Section $1$, we recall statements of the local theory associated to the supercritical SQG (\ref{eq1-supercritical}) and the logarithmically supercritical SQG (\ref{eq1-log}), and establish a preliminary global perturbation bound in the case of supercritical dissipation.  In Section $2$, we use these results to give the proof of Theorem $\ref{thm1}$, our nonlinear instability theorem for the supercritical equation (\ref{eq1-supercritical}).  In Section $3$ we prove Theorem $\ref{thm-log-gwp}$, the global well-posedness results for the logarithmically supercritical SQG, equation (\ref{eq1-log}).  The paper concludes with Section $4$, where we give the proof of Theorem $\ref{thm2}$.

\section{Local well-posedness and global perturbation for supercritical SQG}

We begin with some preliminaries, recalling statements of the local theory for the supercritical SQG (\ref{eq1-supercritical}) and the logarithmically supercritical SQG (\ref{eq1-log})--(\ref{eq-L}).

\begin{proposition}[Local well-posedness for supercritical SQG]
\label{prop-local-supercritical}
Fix $0<\gamma<1$ and let $f\in H^{2-\gamma}(\mathbb{T}^2)$ be given.  For each $\theta_0\in H^{2-\gamma}$, there exists $T=T(\theta_0,f)>0$ such that the initial value problem (\ref{eq1-supercritical}) has a unique local-in-time solution $\theta\in C^0([0,T);H^{2-\gamma}(\mathbb{T}^2))\cap L^2([0,T];H^{2-\gamma/2}(\mathbb{T}^2))$.
\end{proposition}

\begin{proposition}[Local well-posedness for log-supercritical SQG]
\label{prop-local}
Let $0<a<1$, $0<\varepsilon<1$, $1<\tau<5/3$, and $f\in H^{\tau}(\mathbb{T}^2)$ be given.  For each $\theta_0\in H^{\tau}$, there exists $T=T(a,\tau,\varepsilon,\theta_0,f)>0$ such that, for all $\varepsilon>0$, the initial value problem (\ref{eq1-log}) with $\mathcal{L}$ given by (\ref{eq-L}) has a unique local-in-time solution $\theta\in C^0([0,T);H^{\tau}(\mathbb{T}^2))\cap L^2([0,T];H^{\tau+\frac{1}{2}-\varepsilon}(\mathbb{T}^2))$.
\end{proposition}

The proofs of Proposition $\ref{prop-local-supercritical}$ and Proposition $\ref{prop-local}$ are by now relatively standard; we refer the reader to \cite{J06,M06}.

In the rest of this section, we establish an important global estimate for the perturbation equation (\ref{eq-perturbation-supercritical}), which will be used as a tool in the proof of Theorem $\ref{thm1}$.  For $j\in\mathbb{Z}$, define the smoothed Littlewood-Paley projection $$\tilde \Delta_j:=\Delta_{j-1}+\Delta_j+\Delta_{j+1},$$ and the commutator
$$
[f,\Delta_j]g=f\Delta_j g-\Delta_j(fg).
$$

We recall two technical lemmas regarding these operators.

\begin{lemma}
\label{lem2}
Let $\alpha_1<2$ and $\alpha_2<1$ be constants such that
$\alpha_2+\alpha_1>0$. Then for any $f\in H^{\alpha_1}$, $g\in H^{\alpha_2}$, and $j\ge 0$, we have
$$
\|\tilde \Delta_j[f,\Delta_j]g\|_{L^2}\le
C2^{-({\alpha_1}+{\alpha_2}-1)j}c_j\|f\|_{\dot
H^{\alpha_1}}\|g\|_{\dot H^{{\alpha_2}}},
$$
where $\{c_j\}\in l^2$ satisfies $\|c_j\|_{l^2}\leq 1$,
and $C=C(\alpha_1,\alpha_2)$.
\end{lemma}
\begin{proof}
See \cite[Lemma 8.4]{D08} or \cite[Proposition 2]{M06}.
\end{proof}

\begin{lemma}
\label{lem3}
Let ${\alpha_1}<1$ and ${\alpha_2}\in \bR$ be constants. Then for any $f\in H^{\alpha_1}$, $g\in H^{\alpha_2}$, and $j\ge 0$, we have
$$
\|\tilde \Delta_j(fg_j)\|_{L^2}\leq
C2^{-({\alpha_1}+{\alpha_2}-1)j}c_j\|f\|_{\dot
H^{\alpha_1}}\|g\|_{\dot H^{{\alpha_2}}},
$$
where $\{c_j\}\in l^2$ satisfies $\|c_j\|_{l^2}\leq 1$,
and $C=C(\alpha_1,\alpha_2)$.
\end{lemma}
\begin{proof}
See \cite[Lemma 8.5]{D08}.
\end{proof}

Now we give
\begin{proposition}
\label{prop2}
Fix $0<\gamma<1$, and let $f\in H^{3}(\mathbb{T}^2)$ be given with $\int fdx=0$.  Suppose that $\Theta_0\in H^3(\mathbb{T}^2)$ solves the stationary equation (\ref{eq-steady-supercritical}).
Then there exists $C_0>0$ depending on $\gamma$
such that if $\tth$ is a smooth solution to (\ref{eq-perturbation-supercritical}) with
\begin{equation}
\label{eq12.10}
\sup_{t\in [0,\infty)}\|\tth(t,\cdot)\|_{H^{2-\gamma}}\le C_0,
\end{equation}
then we have
\begin{equation}
\label{eq12.11}
\|\tth(t,\cdot)\|_{H^{2-2\gamma/3}}\le C(1+t^{-1/3}),
\end{equation}
for some constant $C>0$ independent of $t$.
\end{proposition}

\begin{proof}
Applying the Littlewood-Paley projection to both sides of the equation
$$
\partial_t \tth+u\cdot \nabla\tth+\tu\cdot\nabla \Theta_0+\Lambda^\gamma \tth=0
$$
with $u=R^\perp\theta$ and $\tu=R^\perp\tth$, we get
$$
\partial_t \tth_j+u\cdot \nabla\tth_j+\tu_j\cdot\nabla \Theta_0+\Lambda^\gamma \tth_j=[u,\Delta_j]\nabla \tth
+[\nabla \Theta_0,\Delta_j]\tu.
$$
We multiply the equation above by $\tth_j$, integrate in $x$, and use $\divop u=0$ to obtain
\begin{align*}
&\partial_t \|\tth_j(t,\cdot)\|_{L^2}+c2^{j\gamma}\|\tth_j(t,\cdot)\|_{L^2} \\
&\le C\big(\|\tilde \Delta_j[u,\Delta_j]\nabla \tth\|_{L^2}
+\|\tilde \Delta_j[\nabla\Theta_0,\Delta_j]\tu \|_{L^2}+\|\tilde \Delta_j(\tu_j\cdot\nabla \Theta_0)\|_{L^2}\big).
\end{align*}
By the Gronwall inequality,
\begin{align*}
& \|\tth_j(t,\cdot)\|_{L^2}\le e^{-c2^{j\gamma}t}\|\tth_j(0,\cdot)\|_{L^2} \\
&\le C\int_0^t e^{-c2^{j\gamma}(t-s)}\big(\|\tilde \Delta_j[u,\Delta_j]\nabla \tth(s,\cdot)\|_{L^2}
+\|\tilde \Delta_j[\nabla \Theta_0,\Delta_j]\tu(s,\cdot)\|_{L^2}\\
&\qquad+\|\tilde \Delta_j(\tu_j\cdot\nabla \Theta_0)(s,\cdot)\|_{L^2}\big)\,ds.
\end{align*}
Now we multiply both sides by $2^{(2-2\gamma/3)j}$ to get
\begin{align*}
& \|\tth_j(t,\cdot)\|_{\dot H^{2-2\gamma/3}}\le Ct^{-1/3}\|\tth_j(0,\cdot)\|_{\dot H^{2-\gamma}} \\
&+ C\int_0^t e^{-c2^{j\gamma}(t-s)}2^{(2-2\gamma/3)j}\big(\|\tilde \Delta_j[u,\Delta_j]\nabla \tth(s,\cdot)\|_{L^2}
+\|\tilde \Delta_j[\nabla \Theta_0,\Delta_j]\tu(s,\cdot)\|_{L^2}\\
&\qquad+\|\tilde \Delta_j(\tu_j\cdot\nabla \Theta_0)(s,\cdot)\|_{L^2}\big)\,ds.
\end{align*}
To estimate the first two terms on the right-hand side, we apply Lemma \ref{lem2} with $\alpha_1=2-2\gamma/3$ and $\alpha_2=1-2\gamma/3$. For the last term, we apply Lemma \ref{lem3} with $\alpha_1=1-2\gamma/3$ and $\alpha_2=2-2\gamma/3$. With a straightforward calculation, we get
\begin{align*}
& \|\tth_j(t,\cdot)\|_{\dot H^{2-\gamma/2}}\le Ct^{-1/3}\|\tth_j(0,\cdot)\|_{\dot H^{2-\gamma}} \\
&\quad +C\int_0^t (t-s)^{-2/3}\|\tth(s,\cdot)\|_{\dot H^{2-2\gamma/3}}\big(\|\theta(s,\cdot)\|_{\dot H^{2-2\gamma/3}}
+\|\Theta_0\|_{\dot H^{2-2\gamma/3}}
\big)\,ds\\
&\le  Ct^{-1/3}\|\tth_j(0,\cdot)\|_{\dot H^{2-\gamma}}+Cc_jt^{-1/3}\Big(\sup_{s\in (0,t)}s^{1/3}\|\tth(s,\cdot)\|_{\dot H^{2-2\gamma/3}}\Big)^2\\
&\quad+Cc_j\|\Theta_0\|_{\dot H^{3-2\gamma/3}}
\sup_{s\in (0,t)}s^{1/3}\|\tth(s,\cdot)\|_{\dot H^{2-2\gamma/3}}.
\end{align*}
Taking the $l_2$ norm on both sides and then taking the supremum in $t\in (0,T)$ for some $T>0$ to be specified, we obtain
\begin{align*}
& \sup_{s\in (0,T)}s^{1/3}\|\tth(s,\cdot)\|_{\dot H^{2-2\gamma/3}}\le C\|\tth(0,\cdot)\|_{\dot H^{2-\gamma}}+C\Big(\sup_{s\in (0,T)}s^{1/3}\|\tth(s,\cdot)\|_{\dot H^{2-2\gamma/3}}\Big)^2\\
&\quad+CT^{1/3}\|\Theta_0\|_{\dot H^{3-2\gamma/3}}
\sup_{s\in (0,T)}s^{1/3}\|\tth(s,\cdot)\|_{\dot H^{2-2\gamma/3}}.
\end{align*}
Recall \eqref{eq12.10}. We then take sufficiently small $T$ and $\varepsilon_0$ to get
$$
\sup_{t\in (0,T)}s^{1/3}\|\tth(t,\cdot)\|_{\dot H^{2-2\gamma/3}}\le 2CC_0.
$$
This gives \eqref{eq12.11} when $t\in (0,T)$. Finally, for $t\ge T$, we can view $t-T/2$ as the initial time.
\end{proof}

\section{Proof of Theorem \ref{thm1}: $(H^{2-\gamma},H^{2-\gamma})$ nonlinear instability for forced supercritical SQG}

In this section we prove Theorem $\ref{thm1}$, which establishes sufficient conditions for $H^{2-\gamma}(\mathbb{T}^2)$ nonlinear instability of the supercritical SQG equation (\ref{eq1-supercritical}).

Let $\gamma$, $f$, and $\Theta_0$ be as in the statement of Theorem $\ref{thm1}$, and let $\mu$ be an eigenvalue of $L$ on $\mathbb{T}^2$ with $\lambda:=\iRe(\mu)$ maximal among all such eigenvalues.  Recall that we have $\lambda>0$ as a consequence of our hypothesis on $\Theta_0$.  Let $\delta>0$ be a small parameter to be determined later in the argument, and set
\begin{align}
L_{\gamma,\delta}=L_\gamma-(\lambda+\delta)I,\label{def-ldelta}
\end{align}
so that $L_{\gamma,\delta}$ is an operator with spectrum entirely in the left half-plane $\iRe z<0$.

It follows from relatively standard arguments that $L_{\gamma,\delta}$ is the generator of an analytic semigroup on $L^2(\mathbb{T}^2)$ (c.f. Lemma 3.1 in \cite{FPV}).  Our first step towards the instability claim is to show a decay estimate for $\varphi\mapsto e^{tL_{\gamma,\delta}}\varphi$.
\begin{lemma}
\label{lemma-ldelta-supercritical}
Fix $\gamma\in (0,1)$, $\varepsilon\in (0,1)$, and $\sigma\in [0,1]$.  Suppose that $\Theta_0\in H^{2+\varepsilon}(\bT^2)$ satisfies $\int \Theta_0=0$ (and thus $R^\perp\Theta_0$ has the same property, e.g. by observing that this property corresponds to vanishing of the zeroth Fourier coefficient).  Then there exists a constant $C>0$, depending on $\gamma$, $\sigma$, and $\lVert \Lambda^{2+\varepsilon}\Theta_0\rVert_{L^2}$, such that for all $t\neq 0$ and $\varphi\in C^{\infty}(\mathbb{T}^2)$ with $\int_{\mathbb{T}^2} \varphi(x)dx=0$ one has
\begin{align}
\Vert e^{tL_{\gamma,\delta}}\varphi\rVert_{L^2}\leq Ct^{-\sigma}\lVert \varphi\rVert_{L^2}^{1-\sigma}\lVert \Lambda^{-\gamma}\varphi\rVert_{L^2}^\sigma,\label{ldelta-bound}
\end{align}
where $L_{\gamma,\delta}$ is as defined in (\ref{def-ldelta}) as a shift of the linearized operator $L_\gamma$.
\end{lemma}

\begin{proof}
We follow the outline of the proof of Lemma $3.2$ in \cite{FPV}.  Fix a parameter $\alpha>0$ and define, for $\varphi\in C^\infty(\mathbb{T}^2)$,
$$A\varphi=A_{\gamma,\alpha}\varphi:=-(R^\perp\Theta_0)\cdot \nabla \varphi-\Lambda^\gamma \varphi-\alpha \varphi.$$

Noting that this can be rewritten as
$$A\varphi=L_{\gamma,\delta}\varphi+(R^\perp \varphi)\cdot\nabla\Theta_0-(\alpha-\lambda-\delta)\varphi,$$
the essential ingredient in the argument is to establish a bound of the form
\begin{align}
\lVert A^{-1}\Lambda^{\gamma}\varphi\rVert_{L^2}\lesssim \lVert\varphi\rVert_{L^2}.\label{eq1-A}
\end{align}
The desired bound (\ref{ldelta-bound}) will then follow from routine application of semigroup decay and interpolation estimates as in \cite{FPV}.

To establish (\ref{eq1-A}), it suffices to show
\begin{align}
\lVert \phi\rVert_{L^2}\lesssim \lVert \Lambda^{-\gamma}A\phi\rVert_{L^2}\label{eq2}
\end{align}
for $\phi$ satisfying $\int \phi dx=0$ (for instance, given $\varphi$ one can take $\phi=A^{-1}\Lambda^\gamma\varphi$; an application of (\ref{eq2}) then gives (\ref{eq1-A})).

To obtain ($\ref{eq2}$), write
\begin{align*}
\int (\Lambda^{-\gamma}A\phi)\overline{\phi}dx=-\int \Big[\Lambda^{-\gamma}((R^\perp\Theta_0)\cdot \nabla \phi)\Big]\overline{\phi}dx-\lVert \phi\rVert_{L^2}^2-\alpha\lVert \Lambda^{-\gamma/2}\phi\rVert_{L^2}^2.
\end{align*}
This leads to the estimate
\begin{align*}
\lVert \phi\rVert_{L^2}^2+\alpha\lVert \Lambda^{-\gamma/2}\phi\rVert_{L^2}^2&=-\int (\Lambda^{-\gamma}A\phi)\overline{\phi} dx-\int \Big[\Lambda^{-\gamma}((R^\perp\Theta_0)\cdot\nabla\phi)\Big]\overline{\phi} dx\\
&\leq \lVert \Lambda^{-\gamma}A\phi\rVert_{L^2}\lVert\phi\rVert_{L^2}+\bigg|\int \Big[(R^\perp\Theta_0)\cdot \nabla\phi\Big]\overline{\Big[\Lambda^{-\gamma}\phi\Big]}dx\bigg|
\end{align*}
Writing
\begin{align*}
(R^\perp\Theta_0)(x)=\sum_{j\in\mathbb{Z}^2, j\neq 0} \widehat{R}_je^{2\pi i j\cdot x},\quad
\phi(x)=\sum_{k\in\mathbb{Z}^2, k\neq 0} \widehat{\phi}_ke^{2\pi i k\cdot x},
\end{align*}
and observing that $\divop(R^\perp\Theta_0)=0$ implies $$\int ((R^\perp\Theta_0)\cdot \nabla \Lambda^{-\gamma/2}\phi)\overline{(\Lambda^{-\gamma/2}\phi)}dx=0,$$
one obtains
\begin{align}
\nonumber &\int \Big[(R^\perp\Theta_0)\cdot\nabla\phi\Big]\overline{\Big[\Lambda^{-\gamma}\phi\Big]}dx\\
\nonumber &\hspace{0.2in}=\int (R^\perp\Theta_0)\cdot\bigg((\nabla\phi)\overline{(\Lambda^{-\gamma}\phi)}-(\nabla \Lambda^{-\gamma/2}\phi)\overline{(\Lambda^{-\gamma/2}\phi)}\bigg)dx\\
\nonumber &\hspace{0.2in}=\int \bigg(\sum_{j\neq 0} \widehat{R}_je^{2\pi i j\cdot x}\bigg)\cdot \bigg[\bigg(\sum_{k\neq 0} (ik)\widehat{\phi}_ke^{2\pi i k\cdot x}\bigg)\bigg(\sum_{\ell\neq 0} |\ell|^{-\gamma}\widehat{\phi}_\ell e^{-2\pi i\ell\cdot x}\bigg)\\
\nonumber &\hspace{0.4in}-\bigg(\sum_{k\neq 0} (ik)|k|^{-\gamma/2}\widehat{\phi}_ke^{2\pi i k\cdot x}\bigg)\bigg(\sum_{\ell\neq 0} |\ell|^{-\gamma/2}\widehat{\phi}_{\ell}e^{-2\pi i\ell\cdot x}\bigg)\bigg]dx\\
&\hspace{0.2in}=\sum_{\substack{j,k,\ell\neq 0\\j+k-\ell=0}} (\widehat{R}_j\cdot ik)(|\ell|^{-\gamma}-|k|^{-\gamma/2}|\ell|^{-\gamma/2})\widehat{\phi}_k\widehat{\phi}_{\ell},\label{eq-rperp}
\end{align}
where we omitted factors of $2\pi$ as usual.
This identity in turn leads to
\begin{align}
\nonumber &\bigg|\int \Big[(R^\perp\Theta_0)\cdot\nabla\phi\Big]\overline{\Big[\Lambda^{-\gamma}\phi\Big]}dx\bigg|\\
&\hspace{0.2in}\leq \sum_{\substack{j,k,\ell\neq 0\\j+k-\ell=0}} |\widehat{R}_j|\,|\widehat{\phi}_k|\, |\widehat{\phi}_\ell|\, \bigg|\frac{|k|}{|\ell|^{\gamma/2}}(|\ell|^{-\gamma/2}-|k|^{-\gamma/2})\bigg|.\label{eq3}
\end{align}

Note that the three vectors $j$, $k$, and $\ell$ form a triangle.  It is easily seen that for any triangle, the length of the largest two sides are comparable.  We discuss three cases.  Set $m=\min \{|j|,|k|,|\ell|\}$.  If $m=|\ell|$, then we have $|j|\sim |k|$, and thus
\begin{align*}
|k||\ell|^{-\gamma/2}\Big||\ell|^{-\gamma/2}-|k|^{-\gamma/2}\Big|\lesssim |k||\ell|^{-\gamma/2}\sim |j||\ell|^{-\gamma/2}.
\end{align*}
Alternatively, if $m=|k|$, then $|\ell|\sim |j|$, and
\begin{align*}
|k||\ell|^{-\gamma/2}\Big||\ell|^{-\gamma/2}-|k|^{-\gamma/2}\Big|\lesssim |k|^{1-\gamma/2}|\ell|^{-\gamma/2}\lesssim |j||\ell|^{-\gamma/2},
\end{align*}
while if $m=|j|$, then $|k|\sim |\ell|$ and by the mean value theorem,
\begin{align*}
|k||\ell|^{-\gamma/2}\Big||\ell|^{-\gamma/2}-|k|^{-\gamma/2}\Big|\lesssim |k||j||\ell|^{-1-\gamma}\leq |j||\ell|^{-\gamma/2}.
\end{align*}

Combining these with the Cauchy-Schwarz inequality, we get, for any $\varepsilon>0$,
\begin{align*}
(\ref{eq3})&\leq C\sum_{j,\ell\neq 0} |\widehat{R}_j|\, |\widehat{\phi}_{\ell-j}|\, |\widehat{\phi}_\ell|\, |j|\, |\ell|^{-\gamma/2}\\
&\leq C\lVert \Lambda^{2+\varepsilon}\Theta_0\rVert_{L^2}\lVert \phi\rVert_{L^2}\lVert \Lambda^{-\gamma/2}\phi\rVert_{L^2}\\
&\leq \frac{1}{2}\lVert \phi\rVert_{L^2}^2+C\lVert \Lambda^{2+\varepsilon}\Theta_0\rVert_{L^2}^2\lVert \Lambda^{-\gamma/2}\phi\rVert_{L^2}^2,
\end{align*}
and thus
\begin{align*}
\frac{1}{2}\lVert \phi\rVert_{L^2}^2+\alpha\lVert \Lambda^{-\gamma/2}\phi\rVert_{L^2}^2\leq \lVert \Lambda^{-\gamma}A\phi\rVert_{L^2}\lVert \phi\rVert_{L^2}+C\lVert \Lambda^{2+\varepsilon}\Theta_0\rVert_{L^2}^2\lVert \Lambda^{-\gamma/2}\phi\rVert_{L^2}^2
\end{align*}
so that if $\alpha$ is chosen sufficiently large (depending on $C$ and $\lVert \Lambda^{2+\varepsilon}\Theta_0\rVert_{L^2}$), we obtain (\ref{eq2}) as desired.
\end{proof}

We now establish Theorem $\ref{thm1}$, showing that linear instability of the stationary solution $\Theta_0$ leads to its nonlinear instability with respect to perturbation in the $H^{2-\gamma}$ norm.

\begin{proof}[Proof of Theorem \ref{thm1}]
Let $\varphi$ be an eigenfunction of $L_\gamma$ associated to the eigenvalue $\mu$, and let $C_0$ be the constant
identified in the statement of Proposition $\ref{prop2}$.  Our aim is to show that there exists $c_0>0$ such that
for all sufficiently small $\varepsilon>0$ the solution $\tth$ to (\ref{eq-perturbation}) evolving from
$\tth(0,\cdot)=\varepsilon \varphi$ eventually has $H^{2-\gamma}$ norm greater than $c_0$.

Fix $\varepsilon>0$, and let $\tth$ be the solution to (\ref{eq-perturbation}) with initial data
$\tth(0)=\varepsilon\varphi$.  Provided $\varepsilon$ is chosen small enough, the local theory for
(\ref{eq1-supercritical}) ensures that $\tth$ is defined at least locally in time.  Since $f\in H^3(\bT^2)$, by
using a bootstrap argument which is similar to the proof of Proposition \ref{prop2}, we have
$\Theta_0\in H^{3+\gamma}(\bT^2)$ and $\varphi\in H^{2+\gamma}(\bT^2)$.

If there exists $t_0>0$ with $\lVert \tth(t_0)\rVert_{H^{2-\gamma}}>C_0$, then we are done, provided our choice
of the constant $c_0$ satisfies $c_0<C_0$.  We may therefore suppose that \eqref{eq12.10} holds.
In view of Proposition $\ref{prop2}$, we have \eqref{eq12.11}.  This combined with the local theory gives
\begin{equation}
\label{eq1.29}
\sup_{t>0}\|\tth(t,\cdot)\|_{H^{2-2\gamma/3}}\le C,
\end{equation}
where $C$ is independent of $\varepsilon$.

By the Duhamel formula, we get
\begin{align}
\nonumber \tth(t)&=\varepsilon e^{tL}\varphi+\int_0^t e^{(t-s)L}N(\tth)(s)ds\\
&=\varepsilon e^{tL}\varphi+\int_0^t e^{(\lambda+\delta)(t-s)}e^{(t-s)L_{\gamma,\delta}}N(\tth)(s)ds\label{eq-duhamel-supercritical}
\end{align}
for all $t\geq 0$.

Now we fix parameters $\sigma\in (0,1)$ and $A>\lVert \varphi\rVert_{L^2}$ to be specified later.  Define $0<T\leq +\infty$ by setting $$T:=\sup\bigg\{\tau>0:\lVert \tth(t)\rVert_{L^2}\leq A\varepsilon e^{\lambda t}\,\, \textrm{for all}\,\, 0\leq t\leq \tau\bigg\},$$ observing that the set inside the supremum is nonempty as a consequence of the continuity of $t\mapsto \tth(t)$ in $L^2$ and the lower bound on the choice of the parameter $A$.

The Minkowski inequality then gives
\begin{align}
\nonumber \varepsilon\lVert e^{tL}\varphi\rVert_{L^2}&\leq \lVert \tth(t)\rVert_{L^2}+\int_0^t e^{(\lambda+\delta)(t-s)}\lVert e^{(t-s)L_{\gamma,\delta}}N(\tth)(s)\rVert_{L^2}\, ds,
\end{align}
so that by Lemma $\ref{lemma-ldelta-supercritical}$ and the definition (\ref{def-n}) of $N(\tth)$, we have
\begin{align*}
\lVert \tth(t)\rVert_{L^2}&\geq \varepsilon\lVert e^{tL}\varphi\rVert_{L^2}-CB(t;\tth)\\
&=\varepsilon e^{\lambda t}\lVert \varphi\rVert_{L^2}-CB(t;\tth),\quad \textrm{for}\quad 0\leq t\leq T,
\end{align*}
where we set
\begin{align}
B(t;\tth):=\int_0^t \frac{e^{(\lambda+\delta)(t-s)}}{(t-s)^{\sigma}}\lVert \tu(s)\cdot \nabla\tth(s)\rVert_{L^2}^{1-\sigma}\lVert \Lambda^{-\gamma}[\tu(s)\cdot\nabla\tth(s)]\rVert_{L^2}^\sigma\, ds\label{def-B-supercritical}
\end{align}
and
\begin{align*}
\tu:=R^\perp\tth.
\end{align*}

Fix $t\in [0,T]$, $s\in (0,t)$, and write $\tu=\tu(s)$, $\nabla\tth=\nabla\tth(s)$.  Then, by H\"older's inequality, the boundedness of Riesz transforms in $L^p$, and the Gagliardo-Nirenberg inequality
\begin{align*}
\|\tth\|_{L^{6/(3-2\gamma)}}&\lesssim \|\tth\|_{\dot{H}^{2\gamma/3}}\lesssim \|\tth\|_{L^2}^{\frac{3-2\gamma}{3-\gamma}}\|\tth\|_{\dot{H}^{2-2\gamma/3}}^{\frac{\gamma}{3-\gamma}},
\end{align*}
we obtain
\begin{align*}
\|\tu\cdot\nabla\tth\|_{L^2}&\le \|\tu\|_{L^{6/(3-2\gamma)}}\|\nabla\tth\|_{L^{3/\gamma}}\\
&\le C\|\tth\|_{L^{6/(3-2\gamma)}}\|\tth\|_{\dot{H}^{2-2\gamma/3}}\\
&\le C\|\tth\|_{L^{2}}^{(3-2\gamma)/(3-\gamma)}\|\tth\|_{\dot H^{2-2\gamma/3}}^{3/(3-\gamma)}
\end{align*}
Similarly, by the fractional Leibniz rule, and the boundedness of Riesz transforms in $L^p$,
we get
\begin{align*}
\|\Lambda^{-\gamma}(\tu\cdot\nabla\tth)\|_{L^2}&=\|\Lambda^{-\gamma}\divop (\tu \tth)\|_{L^2}\\
&\le C\|\Lambda^{1-\gamma}(\tu \tth)\|_{L^2}\\
&\le C\|\tth\|_{L^4}\|\Lambda^{1-\gamma}\tth\|_{L^{4}}\\
&\le C\|\tth\|_{L^2}^{\frac{6-\gamma}{6-2\gamma}}\|\tth\|_{\dot H^{2-2\gamma/3}}^{\frac{6-3\gamma}{6-2\gamma}},
\end{align*}
where we used the Gagliardo-Nirenberg inequalities
\begin{align*}
\|\tth\|_{L^{4}}&\lesssim \|\tth\|_{\dot{H}^{1/2}}\lesssim \|\tth\|_{L^2}^{\frac{9-4\gamma}{12-4\gamma}}\|\tth\|_{\dot{H}^{2-2\gamma/3}}^{\frac{3}{12-4\gamma}},\\
\|\Lambda^{1-\gamma}\tth\|_{L^{4}}&\lesssim \|\tth\|_{\dot{H}^{(3/2)-\gamma}}\lesssim \|\tth\|_{L^2}^{\frac{2\gamma+3}{12-4\gamma}}\|\tth\|_{\dot{H}^{2-2\gamma/3}}^{\frac{9-6\gamma}{12-4\gamma}}.
\end{align*}

We now choose $\sigma$ sufficiently close to $1$ to get
$$
\|\tu\cdot \nabla\tth\|_{L^2}^{1-\sigma}\|\Lambda^{-\gamma}(\tu\cdot \nabla\tth)\|_{L^2}^{\sigma}
\le C\|\tth\|_{L^2}^{1+\beta}\|\tth\|_{\dot H^{2-2\gamma/3}}^{1-\beta}.
$$
for some $\beta\in (0,1)$ (for instance, one can choose $\sigma=3/4$, leading to $\beta=\gamma/(24-8\gamma)$).
Combining this bound with the definition (\ref{def-B-supercritical}) of $B(t;\tth)$, imposing the condition $0<\delta<\lambda\beta/2$ on $\delta$, using \eqref{eq1.29}, and recalling that by our choice of $T$ we have $\|\tth(s)\|_{L^2}\leq A\varepsilon e^{\lambda s}$ for $0\leq s\leq t\leq T$, we get
\begin{align}
B(t;\tth)&\leq C_1 \int_0^t \frac{e^{\lambda(1+\beta)(t-s)}}{(t-s)^{\sigma}}\|\tth(s)\|_{L^2}^{1+\beta}e^{-\frac{\lambda\beta}{2}(t-s)}\, ds
\leq C_1(A\varepsilon e^{\lambda t})^{1+\beta}\label{eq-B-estimate-supercritical}
\end{align}
for $0\leq t\leq T$.
It follows that if we choose $C_2>0$ such that $t_*=\lambda^{-1}\log(C_2/\varepsilon)$ satisfies $0\leq t_*\leq T$, then we have
\begin{align}
\nonumber \lVert \tth(t_*)\rVert_{L^2}&\geq \varepsilon e^{\lambda t_*}\lVert\varphi\rVert_{L^2}-C_1(A\varepsilon e^{\lambda t_*})^{1+\beta}\\
&=C_2(\lVert\varphi\rVert_{L^2}-A^{1+\beta}C_1C_2^{\beta}).\label{eq-tt-below}
\end{align}

To choose $C_2$, we first consider the case when $T<+\infty$ and identify a lower bound on $T$.  Note that the continuity of $t\mapsto \lVert \tth(t)\rVert_{L^2}$ gives $\lVert \tth(T)\rVert_{L^2}=A\varepsilon e^{\lambda T}$.  The Duhamel formula (\ref{eq-duhamel-supercritical}) and the estimate (\ref{eq-B-estimate-supercritical}) then imply
\begin{align*}
A\varepsilon e^{\lambda T}=\lVert \tth(T)\rVert_{L^2}\leq \varepsilon e^{\lambda T}\lVert \varphi\rVert_{L^2}+B(T;\tth)\leq \varepsilon e^{\lambda T}\lVert \varphi\rVert_{L^2}+C_1(A\varepsilon e^{\lambda T})^{1+\beta}
\end{align*}
so that
\begin{align*}
e^{\lambda T}\geq \frac{1}{\varepsilon}\left(\frac{A-\lVert \varphi\rVert_{L^2}}{C_1A^{1+\beta}}\right)^{1/\beta}
\end{align*}
i.e. $$T\geq \frac{1}{\lambda}\log\left(\frac{1}{\varepsilon}\left[\frac{A-\lVert \varphi\rVert_{L^2}}{C_1A^{1+\beta}}\right]^{1/\beta}\right).$$
The above inequality also holds trivially when $T=+\infty$.
Therefore, choosing
\begin{align*}
C_2:=\left[\frac{A-\lVert \varphi\rVert_{L^2}}{C_1A^{1+\beta}}\right]^{1/\beta},
\end{align*}
we have $0\leq t_*\leq T$, and thus, in view of (\ref{eq-tt-below}),
\begin{align*}
\lVert \tth(t_*)\rVert_{L^2}
&\geq C_2(2\lVert \varphi\rVert_{L^2}-A)\\
&=\bigg(\frac{A-\lVert \varphi\rVert_{L^2}}{C_1A^{1+\beta}}\bigg)^{1/\beta}(2\lVert \varphi\rVert_{L^2}-A)
\end{align*}
in this case.  The choice $A=3\lVert \varphi\rVert_{L^2}/2$ now leads to
\begin{align*}
\lVert \tth(t_*)\rVert_{L^2}\geq \left[\frac{1}{3^{1+\beta}C_1}\right]^{1/\beta}.
\end{align*}
This completes the argument.
\end{proof}

\section{Global well-posedness for log-supercritical SQG}

In this section, we begin our detailed study of long-time properties of the logarithmically supercritical SQG evolution (\ref{eq1-log}), with nonlocal diffusion given by the operator $\mathcal{L}$ defined in ($\ref{eq-L}$).  We begin by remarking that this operator can be expressed in several alternative ways, all equivalent up to multiplication by a fixed dimensional constant.  In particular, we make note of the convolution representation
\begin{align}
(\mathcal{L}\theta)(x)=\int_{\mathbb{R}^2} (\theta(x)-\theta(x+y))K(y)dy,\label{eqK}
\end{align}
where the integral is interpreted in the principal value sense, and with kernel $K$ satisfying
\begin{align*}
|K(y)|\lesssim |y|^{-3}\log^{-a}(\kappa+|y|^{-1}),
\end{align*}
\begin{align*}
|(\nabla K)(y)|\lesssim |y|^{-4}\log^{-a}(\kappa+|y|^{-1}),
\end{align*}
for $y\in\mathbb{R}^2\setminus\{0\}$, and
\begin{align*}
K(y)\gtrsim |y|^{-3}\log^{-a}(|y|^{-1})
\end{align*}
for $|y|\in (0,r_0)$, where $r_0>0$ is a small constant depending only on $a$. See, for instance, \cite{DKSV}.

Note that smooth solutions  of (\ref{eq1-log}) satisfy a number of a priori bounds (and thus, via an approximation argument, the same is true for the local solutions constructed in Proposition \ref{prop-local}).  In particular,
\begin{align*}
\sup_{t\in [0,T)}\lVert \theta(t)\rVert_{L^\infty}\leq M_{\theta,f}:=C(\lVert \theta_0\rVert_{L^\infty}+\lVert f\rVert_{L^\infty})
\end{align*}
See also Lemma $5.4$ in \cite{DKSV}, where Fourier transform considerations are used to show that bounds of similar type hold for diffusion given by a suitably chosen Fourier multiplier, even when the associated kernel is not necessarily positive.

Our presentation of the global existence theory is motivated by the method of nonlinear lower bounds formulated in \cite{CV}.  This approach was used to study the forced critical SQG in \cite{CTV}, and developed in a supercritical context (with an additional decay factor which allows to exploit ``eventual regularization'' properties of the equation) in \cite{CZV}.  Some results related to this approach were also obtained in \cite{SGS,MX,XY}.

Fix $\xi_0$ and let $\xi:[0,\infty)\rightarrow [0,\infty)$ be a decreasing function with $\xi(0)=\xi_0$.  Both $\xi_0$ and $\xi$ will be specified further later in the argument (see Section $1.2$ below).  For $(t,x)\in [0,T]\times \mathbb{T}^2$ and $h\in \mathbb{T}^2$, define
\begin{align}
v(t,x;h)&:=(\xi(t)^2+|h|^2)^{-\alpha/2}(\theta(t,x+h)-\theta(t,x))\label{eq-def-v}\\
F(t,x;h)&:=(\xi(t)^2+|h|^2)^{-\alpha/2}(f(x+h)-f(x)),\nonumber 
\end{align}
where $\alpha\in (0,1)$ is a small constant to be specified.

\subsection*{Equation for $v^2$ and associated bounds}

To identify a suitable equation for $v$ (in fact, an equation for $v^2$), we first recall a pointwise identity for $\mathcal{L}$,
\begin{align}
\mathcal{L}(g^2)=2g\mathcal{L}g-c\int_{\mathbb{R}^2} (g(x)-g(x+y))^2K(y) \ dy,\label{ptwise-L}
\end{align}
for all $g\in C^\infty(\mathbb{T}^2)$, where $K$ is the convolution kernel associated to the representation (\ref{eqK}), and where the integral is interpreted in the principal value sense.

Now, setting
\begin{align}
w(t,x;h)&:=(R^\perp\theta)(t,x+h)-(R^\perp\theta)(t,x)\label{eq-def-w}
\end{align}
and
\begin{align*}
D_h[v(t)]&:=(2\pi)^{-1}\int_{\mathbb{R}^2} (v(t,x;h)-v(t,x+y;h))^2K(y) \ dy,
\end{align*}
and using the identity (\ref{ptwise-L}), we obtain that $v^2$ solves
\begin{align}
\nonumber &\partial_t v^2(t,x;h)+(R^\perp\theta)(t,x)\cdot \nabla_xv^2(t,x;h)\\
\nonumber &\hspace{0.2in}+w(t,x;h)\cdot\nabla_h v^2(t,x;h)+\mathcal{L}v^2+D_h[v(t)]\\
&\hspace{0.4in}=-\frac{2\alpha \xi(t)\xi'(t)}{\xi(t)^2+|h|^2}v^2(t,x;h)+\frac{2\alpha h\cdot w(t,x;h)}{\xi(t)^2+|h|^2}v^2(t,x;h)+2v(t,x;h)F(t,x;h).\label{eq-v}
\end{align}

\begin{lemma}
\label{lem-lower-bd}
For each $a\in (0,1/2)$, there exist constants $c>0$ and $C>0$ such that if $\theta$ is a smooth solution to (\ref{eq1-log}), with $v$ and $w$ defined as in (\ref{eq-def-v}) and (\ref{eq-def-w}), then for every $(t,x)\in [0,T]\times\mathbb{T}^2$ and $h\in\mathbb{R}^2$, one has the estimates
\begin{align}
\nonumber |w(t,x;h)|&\leq C_a(\xi(t)^2+|h|^2)^{\alpha/2}(R\log^a(\kappa+R^{-1}))^{1/2}(D_h[v(t)])^{1/2}\\
&\hspace{0.2in}+C\lVert v\rVert_{L_{t,x,h}^\infty}|h|(\xi(t)^\alpha/R+R^{-(1-\alpha)}),\label{eq-lower-bd-1}
\end{align}
and
\begin{align}
\nonumber (D_h[v(t)])(x)&\geq \frac{c|v(t,x;h)|^2}{R\log^a(\kappa+R^{-1})}
-\frac{C\|v\|_{L^\infty}^2}{r_0\log^a(\kappa+r_0^{-1})}\\
&\hspace{0.4in}-\frac{C\lVert v\rVert_{L^\infty}|h|\,|v(t,x;h)|}{(\xi(t)^2+|h|^2)^{\alpha/2}}(\xi(t)^\alpha/R^2+R^{-(2-\alpha)})\label{eq-lower-bd-2}
\end{align}
for all $R>0$ with $R\geq 4|h|$.
\end{lemma}

\begin{proof}
Let $\chi\in C^\infty([0,\infty);[0,\infty))$ be a fixed cutoff function, decreasing on $[0,\infty)$, and satisfying $\chi(x)=1$ for $x\in [0,1]$, $\supp\chi\subset [0,2)$, and $\lVert \chi'\rVert_{L^\infty}\leq 2$.  For $R>0$, set $\chi_R(y)=\chi(|y|/R)$.

We begin by showing (\ref{eq-lower-bd-1}).  For this, we write
\begin{align*}
w(t,x;h)&=\int_{\mathbb{R}^2} \frac{y^\perp}{|y|^3}\Big[\theta(t,x+y+h)-\theta(t,x+y)\Big] \ dy\\
&=(\xi(t)^2+|h|^2)^{\alpha/2}\int_{\mathbb{R}^2} \frac{y^\perp}{|y|^3}\Big[v(t,x+y;h)-v(t,x;h)\Big] \ dy\\
&=(\xi(t)^2+|h|^2)^{\alpha/2}\int_{\mathbb{R}^2} \frac{y^\perp}{|y|^3}\Big[v(t,x+y;h)-v(t,x;h)\Big]\chi_R(y) \ dy\\
&\hspace{0.2in}+(\xi(t)^2+|h|^2)^{\alpha/2}\int_{\mathbb{R}^2} \frac{y^\perp}{|y|^3}\Big[v(t,x+y;h)-v(t,x;h)\Big](1-\chi_R(y)) \ dy\\
&=:I_1+I_2,
\end{align*}
where we exploited the cancellation properties of this singular integral (via the odd symmetry of the kernel $y^\perp/|y|^3$).  We then obtain
\begin{align*}
&\int_{\mathbb{R}^2} \frac{y^\perp}{|y|^3}\Big[v(t,x+y;h)-v(t,x;h)\Big]\chi_R(y) \ dy\\
&\hspace{0.2in}\lesssim \left(\int_{0}^R \log^a(\kappa+r^{-1}) \ dr\right)^{1/2}(D_h[v(t)])^{1/2},
\end{align*}
so that, since
\begin{align*}
\int_0^R \log^a(\kappa+r^{-1}) \ dr\leq \frac{1}{\log^{1-a}(\kappa+R^{-1})}\int_0^R \log(\kappa+r^{-1}) \ dr\lesssim R\log^a(\kappa+R^{-1}),
\end{align*}
we have
\begin{align}
|I_1|&\leq C_a(\xi(t)^2+|h|^2)^{\alpha/2}(R\log^a(\kappa+R^{-1}))^{1/2}(D_h[v(t)])^{1/2}.\label{eq-w-bound-1}
\end{align}
On the other hand, using the change of variables $y\mapsto y-h$ and setting $G(y):=\frac{y^\perp}{|y|^3}(1-\chi_R(y))$ for $y\in\mathbb{R}^2$,
\begin{align*}
I_2=\int_{\mathbb{R}^2} (G(y-h)-G(y))(\xi(t)^2+|y|^2)^{\alpha/2}v(t,x;y) \ dy.
\end{align*}
We now invoke the mean value theorem to estimate
\begin{align*}
|G(y-h)-G(y)|&\leq |h|\sup_{z\in \{y-sh:s\in [0,1]\}} |\nabla G(z)|\\
&\lesssim |h|\sup_{s\in [0,1]} |y-sh|^{-3}\chi_{\{y:|y-sh|\geq R\}}(y).
\end{align*}

Since $R$ satisfies $R\geq 4|h|$ (so that for $0\leq s\leq 1$, $\{y:|y-sh|\geq R\}$ nonempty implies $|y|\geq 3R/4$), we obtain
\begin{align}
\nonumber |I_2|&\lesssim |h|\int_{|y|\geq 3R/4} \frac{(\xi(t)^2+|y|^2)^{\alpha/2}}{|y|^3}|v(t,x;y)| \ dy\\
&\lesssim \lVert v\rVert_{L_{t,x,h}^\infty}|h|(\xi(t)^\alpha R^{-1}+R^{-1+\alpha}).\label{eq-w-bound-2}
\end{align}
Imposing this assumption on $R$, the desired inequality (\ref{eq-lower-bd-1}) now follows by combining (\ref{eq-w-bound-1}) and (\ref{eq-w-bound-2}).

It remains to show (\ref{eq-lower-bd-2}), for which we use a similar argument.  Fix $R>0$, and note that
\begin{align}
\nonumber &\frac{\|v\|_{L^\infty}^2}{r_0\log^a(\kappa+|r_0|^{-1})} + (D_h[v(t)])(x)\\
\nonumber &\gtrsim \int_{\mathbb{R}^2} \frac{(v(t,x;h)-v(t,x+y;h))^2}{|y|^{3}\log^{a}(\kappa+|y|^{-1})} \ dy\\
\nonumber &\gtrsim \int_{\mathbb{R}^2} \frac{(v(t,x;h)-v(t,x+y;h))^2}{|y|^3\log^a(\kappa+|y|^{-1})}\bigg(1-\chi_R(y)\bigg) \ dy\\
\nonumber &\gtrsim v(t,x;h)^2\int_{2R}^\infty \frac{1}{r^2\log^a(\kappa+r^{-1})} \ dr,\\
&\hspace{0.2in}-2|v(t,x;h)|\,\bigg|\int_{\mathbb{R}^2} \frac{v(t,x+y;h)}{|y|^3\log^a(\kappa+|y|^{-1})}\bigg(1-\chi_R(y)\bigg) \ dy\bigg|.\label{eq-lower-bound-3}
\end{align}

Using the change of variables $y\mapsto y-h$ as before (and recalling the definition of $v$ given by (\ref{eq-def-v})), the second term appearing on the right-hand side of (\ref{eq-lower-bound-3}) is bounded from above by a multiple of
\begin{align}
\nonumber &\frac{|v|}{(\xi(t)^2+|h|^2)^{\alpha/2}}\bigg|\int_{\mathbb{R}^2} \theta(t,x+y)(H(y-h)-H(y)) \ dy\bigg|\\
\nonumber &\hspace{0.2in}=\frac{|v|}{(\xi(t)^2+|h|^2)^{\alpha/2}}\bigg|\int_{\mathbb{R}^2} (\theta(t,x+y)-\theta(t,x))(H(y-h)-H(y)) \ dy\bigg|\\
\nonumber &\hspace{0.2in}=\frac{|v|}{(\xi(t)^2+|h|^2)^{\alpha/2}}\bigg|\int_{\mathbb{R}^2} v(t,x;y)(\xi(t)^2+|y|^2)^{\alpha/2}(H(y-h)-H(y)) \ dy\bigg|\\
&\hspace{0.2in}\leq \frac{|v|\,\lVert v(t)\rVert_{L^\infty}}{(\xi(t)^2+|h|^2)^{\alpha/2}}\int_{\mathbb{R}^2} (\xi(t)^2+|y|^2)^{\alpha/2}|H(y-h)-H(y)| \ dy,\label{eq-rhs1}
\end{align}
where we set $H(y):=\frac{1}{|y|^3\log^a(\kappa+|y|^{-1})}(1-\chi_R(y))$ for $y\in\mathbb{R}^2$.  Again invoking
the mean value theorem to estimate
$$|H(y-h)-H(y)|\leq |h|\sup_{z\in\{y-sh:s\in [0,1]\}} |\nabla H(z)|,$$
we bound the right-hand side of (\ref{eq-rhs1}) by
\begin{align}
\nonumber &\frac{|h|\, |v|\, \lVert v(t)\rVert_{L^\infty}}{(\xi(t)^2+|h|^2)^{\alpha/2}}\int_{\mathbb{R}^2} (\xi(t)^2+|y|^2)^{\alpha/2}\sup_{s\in [0,1]}\bigg(\frac{\chi_{\{y:R\leq |y|\leq 2R\}}(y-sh)}{R|y-sh|^3\log^a(\kappa+|y-sh|^{-1})}\\
\nonumber &\hspace{2.2in}+ \frac{\chi_{\{y:|y|\geq R\}}(y-sh)}{|y-sh|^4\log^a(\kappa+|y-sh|^{-1})} \bigg) \ dy\\
\nonumber &\lesssim \frac{|h|\, |v|\, \lVert v(t)\rVert_{L^\infty}}{(\xi(t)^2+|h|^2)^{\alpha/2}}\bigg(\int_{3R/4}^{9R/4} \frac{(\xi(t)^2+r^2)^{\alpha/2}}{Rr^2\log^a(\kappa+r^{-1})} \ dr+\int_{3R/4}^{\infty}\frac{(\xi(t)^2+r^2)^{\alpha/2}}{r^3\log^a(\kappa+r^{-1})} \ dr\bigg)\\
\nonumber &\lesssim |h|\, |v|\, \lVert v\rVert_{L^\infty}\, (\xi(t)^2+|h|^2)^{-\alpha/2}(\xi(t)^\alpha/R^2+R^{-(2-\alpha)}).
\end{align}
Assembling these estimates completes the proof of the desired bound (\ref{eq-lower-bd-2}).
\end{proof}

We next elaborate on the lower bound ($\ref{eq-lower-bd-2}$) for $D_h[v(t)]$, by making an appropriate choice of $R\geq 4|h|$.  With $C,c$ as in the statement of Lemma $\ref{lem-lower-bd}$, choosing $$R=\bigg(\frac{4C\lVert v\rVert_{L^\infty}}{c|v(t,x;h)|}\bigg)^{1/(1-\alpha)}|h|\log^{a/(1-\alpha)}(\kappa+(4|h|)^{-1})$$ (possibly increasing $C$ to ensure $(4C/c)^{1/(1-\alpha)}\geq 4$ and thus $R\geq 4|h|$), we have
\begin{align*}
\frac{\xi(t)^\alpha}{R}&=\left(\frac{c^{1/(1-\alpha)}\xi(t)^{\alpha}}{(4C)^{1/(1-\alpha)}|h|\log^{a/(1-\alpha)}(\kappa+(4|h|)^{-1})}\right)\cdot\left(\frac{|v|}{\lVert v\rVert_{L^\infty}}\right)^{1/(1-\alpha)}\\
&\leq \frac{c\xi(t)^{\alpha}}{4C|h|\log^a(\kappa+(4|h|)^{-1})}\cdot \frac{|v|}{\lVert v\rVert_{L^\infty}}\\
&\leq \frac{c(\xi(t)^2+|h|^2)^{\alpha/2}}{4C|h|\log^a(\kappa+(4|h|)^{-1})}\cdot \frac{|v|}{\lVert v\rVert_{L^\infty}},
\end{align*}
along with
\begin{align*}
\frac{1}{R^{1-\alpha}}&=\frac{c}{4C|h|^{1-\alpha}\log^{a}(\kappa+(4|h|)^{-1})}\cdot\left(\frac{|v|}{\lVert v\rVert_{L^\infty}}\right)\\
&\leq\frac{c|h|^{\alpha}}{4C|h|\log^{a}(\kappa+(4|h|)^{-1})}\cdot\left(\frac{|v|}{\lVert v\rVert_{L^\infty}}\right)\\
&\leq\frac{c(\xi(t)^2+|h|^2)^{\alpha/2}}{4C|h|\log^{a}(\kappa+(4|h|)^{-1})}\cdot\left(\frac{|v|}{\lVert v\rVert_{L^\infty}}\right),
\end{align*}
and thus
\begin{align*}
\frac{C\lVert v\rVert_{L^\infty}|h|\,|v(t,x;h)|}{R(\xi(t)^2+|h|^2)^{\alpha/2}}(\xi(t)^\alpha/R+R^{-(1-\alpha)})\leq \frac{c|v(t,x;h)|^2}{2R\log^a(\kappa+(4|h|)^{-1})}.
\end{align*}

It then follows that the bound (\ref{eq-lower-bd-2}) gives
\begin{align*}
(D_h[v(t)])(x)
&\geq \frac{c|v(t,x;h)|^2}{2R\log^a(\kappa+(4|h|)^{-1})}
-\frac{C\|v\|_{L^\infty}^2}{r_0\log^a(\kappa+r_0^{-1})},
\end{align*}
which in turn implies
\begin{align}
(D_h[v(t)])(x)&\geq c_1\bigg(\frac{ |v(t,x;h)| }{ \lVert v\rVert_{L^\infty} }\bigg)^{1/(1-\alpha)}\frac{|v(t,x;h)|^2}{|h|\log^{a(2-\alpha)/(1-\alpha)}(\kappa+(4|h|)^{-1})}\label{eq-lower-bd-3}\nonumber\\
&\quad -\frac{C\|v\|_{L^\infty}^2}{r_0\log^a(\kappa+r_0^{-1})}
\end{align}
for some $c_1>0$ depending only on $a$.

\subsection*{A differential inequality for $\xi$}

Fix $\alpha\in (0,1/2)$, and let $\varphi:[0,\infty)\rightarrow [0,\infty)$ be defined by
\begin{align*}
\varphi(x)=\int_0^x \log^{a\frac{2-\alpha}{1-\alpha}}(\kappa+(4s)^{-1})ds.
\end{align*}
Then $\varphi$ is continuous and strictly increasing on $[0,\infty)$, with $\varphi(0)=0$ and $$\lim_{x\rightarrow\infty} \varphi(x)=+\infty,$$ and so has a well-defined inverse $\varphi^{-1}$.

Now, fix $c_0>0$ (to be determined later in the argument, depending only on $a$) and $\xi_0>0$, set $$T_*:=4\alpha\varphi(\xi_0)/c_0,$$ and define $\xi:[0,\infty)\rightarrow [0,\infty)$ by
\begin{align*}
\xi(t)=\varphi^{-1}\left(\varphi(\xi_0)-\frac{c_0}{4\alpha}t\right)
\end{align*}
for $0\leq t<T_*$, and $\xi(t)=0$ for $t\geq T_*$.

Then $\xi$ is a decreasing function with $\xi(0)=\xi_0$ which solves the differential inequality
\begin{align*}
|\xi'|\leq \frac{c_0}{4\alpha\log^{a\frac{2-\alpha}{1-\alpha}}(\kappa+(4|\xi|)^{-1})}.
\end{align*}
These properties will play an important role in the arguments below.

\subsection*{The main a priori bound: preserving $C^\alpha$ estimates}

We now state and prove the main global a priori estimate, which is a bound on the H\"older seminorm of smooth solutions.  We first record an elementary inequality related to the logarithmic lower bound in (\ref{eq-lower-bd-3}).

\begin{lemma}
\label{lem-alm-mon}
There exists $C>0$ such that for every $\gamma>0$ and $s,t\in (0,\infty)$ with $s<t$, one has
\begin{align*}
s^{\gamma}\log(\kappa+(4s)^{-1})\leq (C/\gamma)t^\gamma\log(\kappa+(4t)^{-1}).
\end{align*}
\end{lemma}

We next state and prove the a priori H\"older bound.

\begin{proposition}
\label{prop4}
Let $0<a<1/2$ be given.  Then there exists $X_0=X_0(a)\in (0,1)$ such that the following statement holds.

For each $M_{\theta,f}>0$ and $\xi_0\in (0,X_0)$, there exists $\alpha_0=\alpha_0(a,M_{\theta,f},\xi_0)\in (0,1/2)$ such that if $f\in C^\infty(\mathbb{T}^2)$ satisfies $$\lVert \theta_0\rVert_{L^\infty}+\lVert f\rVert_{L^\infty}\leq M_{\theta,f},$$
and $\theta$ is a smooth solution to (\ref{eq1-log}) on a time interval $[0,T_*]$, $T_*=4\alpha\varphi(\xi_0)/c_0$, with $\varphi$ as defined in Section 1.2,
then
$$|\theta(t)|_{C^\alpha}\lesssim C(M_{\theta,f},\xi_0)$$
for $t\geq T_*$.
\end{proposition}

\begin{proof}
Let $v$ be as defined in (\ref{eq-def-v}).  Define also $$M:=\frac{4M_{\theta,f}}{\xi_0^\alpha},$$ and note that $$\lVert v(0)\rVert_{L^\infty}\leq \frac{2\lVert \theta_0\rVert_{L^\infty}}{\xi_0^\alpha}\leq \frac{M}{2}.$$

Now, set $T=\sup\{t_0:\lVert v(t)\rVert_{L^\infty}<M$ for all $0<t<t_0\}$.  We want to show that $T=+\infty$, and will argue by contradiction.  Suppose that $T$ is finite, and choose $t_0\in (0,T)$ such that for all $t\in (t_0,T)$ one has
$$\lVert v(t)\rVert_{L^\infty}>\frac{3M}{4}, \quad \lVert v(t_0)\rVert_{L^\infty}=\frac{3M}{4}.$$

For $t\in [t_0,T)$, define
\begin{align*}
g(t):=\sup_{x\in\mathbb{T}^2,\, h\in\mathbb{R}^2} |v(t,x;h)|^2,
\end{align*}
and choose $x_0(t), h_0(t)\in\mathbb{T}^2$ such that $$g(t)=v(t,x_0(t);h_0(t))^2$$ and $$g'(t)=(\partial_t v^2)(t,x_0(t);h_0(t)).$$

In what follows, we fix $t\in [t_0,T)$, and set $x_0=x_0(t)$ and $h_0=h_0(t)$.  Note that because $x\in \mathbb{T}^2$ we immediately have $|h_0|\leq 4\pi$, while the observation that $|h|>\xi_0$ implies
\begin{align*}
|v(t,x_0,h)|\leq \frac{2\lVert \theta(t)\rVert_{L^\infty}}{(\xi(t)^2+|h|^2)^{\alpha/2}}\leq \frac{2\lVert\theta_0\rVert_{L^\infty}}{|h|^\alpha}\leq \frac{2\lVert \theta_0\rVert_{L^\infty}}{\xi_0^\alpha}\leq \frac{M}{2},
\end{align*}
which shows that the choice of $(x_0,h_0)$ gives $|h_0|\leq \xi_0$ (since the restriction $t\geq t_0$ gives $v(t;x_0,h_0)=\sqrt{g(t)}\geq \frac{3M}{4}$).

By the choice of $x_0$ and $h_0$, we obtain the optimality conditions
$$\nabla_x v^2(t,\cdot;h_0(t))|_{x=x_0(t)}=0,\quad
\nabla_h v^2(t,x_0(t);\cdot)|_{h=h_0(t)}=0,$$
and
$$\mathcal{L}v^2(t,\cdot,h_0(t))|_{x=x_0(t)}\geq 0.$$
Combining these with the equation (\ref{eq-v}) for $v^2$, we obtain
\begin{align}
g'(t)+D_{h_0}[v(t)]&\leq \frac{2\alpha |\xi(t)\xi'(t)|}{\xi(t)^2+|h_0|^2}v^2+\frac{2\alpha h_0\cdot w}{\xi(t)^2+|h_0|^2}v^2+2vF,\label{eq-gprime-1}
\end{align}
where we omitted the evaluation at $(t,x_0;h_0)$ when no ambiguity can arise.

Suppose that $\xi$ is a non-negative decreasing function with $\xi(0)=\xi_0$ as given above, in particular solving the differential inequality
\begin{align*}
|\xi'|\leq \frac{c_0}{4\alpha\log^{a\frac{2-\alpha}{1-\alpha}}(\kappa+(4\xi)^{-1})},
\end{align*}
and set $$d_{h_0}[v]=d_{h_0}[v(t)](x_0):=\frac{c_0}{|h_0|\log^{a(2-\alpha)/(1-\alpha)}(\kappa+(4|h_0|)^{-1})}v^2.$$

We then have
\begin{align*}
&\frac{2\alpha |\xi\xi'|}{\xi^2+|h_0|^2}v^2\\
&\hspace{0.2in}\leq \frac{c_0\xi}{(\xi^2+|h_0|^2)\log^{a\frac{2-\alpha}{1-\alpha}}(\kappa+(4\xi)^{-1})}v^2\\
&\hspace{0.2in}\leq \frac{c_0(\xi^2+|h_0|^2)^{1/2}}{(\xi^2+|h_0|^2)\log^{a\frac{2-\alpha}{1-\alpha}}(\kappa+(4(\xi^2+|h_0|^2)^{1/2})^{-1})}v^2\\
&\hspace{0.2in}=\frac{c_0}{(\xi^2+|h_0|^2)^{1/2}\log^{a\frac{2-\alpha}{1-\alpha}}(\kappa+(4(\xi^2+|h_0|^2)^{1/2})^{-1})}v^2\\
&\hspace{0.2in}\leq C_2d_{h_0}[v],
\end{align*}
where $C_2>0$ depends only on $a$, and where to obtain the last inequality we used Lemma $\ref{lem-alm-mon}$.

We now estimate the second term on the right-hand side of (\ref{eq-gprime-1}).  Using the first estimate in Lemma $\ref{lem-lower-bd}$, we obtain
\begin{align*}
&\frac{2\alpha h_0\cdot w}{\xi(t)^2+|h_0|^2}v^2\\
&\hspace{0.2in}\leq \frac{2C\alpha|h_0|(R_1\log^a(\kappa+R_1^{-1}))^{1/2}}{(\xi^2+|h_0|^2)^{1-(\alpha/2)}}|D_{h_0}[v(t)]|^{1/2}|v|^2\\
&\hspace{1.2in}+\frac{2\alpha C|h|^2\lVert v\rVert_{L^\infty}}{\xi^2+|h_0|^2}\bigg(\frac{\xi^\alpha}{R_1}+\frac{1}{R_1^{1-\alpha}}\bigg)|v|^2
\end{align*}
for all $R_1\geq 4|h_0|$, so that an application of Young's inequality gives the bound
\begin{align*}
&\frac{1}{2}|D_{h_0}[v(t)]|+C\bigg(\frac{\alpha^2|h_0|^2}{(\xi^2+|h_0|^2)^{2-\alpha}}R_1\log^a(\kappa+R_1^{-1})|v|^4\\
&\hspace{1.6in}+\frac{\alpha|h_0|^2}{\xi^2+|h_0|^2}\lVert v\rVert_{L^\infty}\bigg(\frac{\xi^\alpha}{R_1}+\frac{1}{R_1^{1-\alpha}}\bigg)|v|^2\bigg)
\end{align*}
for all such $R_1$, which is equal to
\begin{align*}
&\frac{1}{2}|D_{h_0}[v(t)]|+\frac{C\alpha|h_0|^2}{\xi^2+|h_0|^2}\bigg(\frac{\alpha R_1\log^a(\kappa+R_1^{-1})|v|^2}{(\xi^2+|h_0|^2)^{1-\alpha}}\\
&\hspace{2.2in}+\frac{\lVert v\rVert_{L^\infty}\xi^\alpha}{R_1}+\frac{\lVert v\rVert_{L^\infty}}{R_1^{1-\alpha}}\bigg)|v|^2.
\end{align*}

Now, choosing $R_1=4(\xi^2+|h_0|^2)^{1/2}$ and recalling that we assumed $\lVert v\rVert_{L^\infty}\leq M=4M_{\theta,f}/\xi_0^\alpha$, we obtain
\begin{align}
\nonumber &\frac{\alpha R_1\log^a(\kappa+R_1^{-1})|v|^2}{(\xi^2+|h_0|^2)^{1-\alpha}}\\
\nonumber &\hspace{0.2in}\leq\frac{64M_{\theta,f}^2\alpha \log^a(\kappa+(4(\xi^2+|h_0|^2)^{1/2})^{-1})}{(\xi^2+|h_0|^2)^{(1/2)-\alpha}\xi_0^{2\alpha}}\\
&\hspace{0.2in}\leq \frac{CM_{\theta,f}^2\alpha^{1-\beta}}{|h_0|\log^{a\frac{2-\alpha}{1-\alpha}}
(\kappa+(4|h_0|)^{-1})}\log^{\beta}(\kappa+(4\xi_0)^{-1}),\label{eq-R1term}
\end{align}
where we set $\beta=a\frac{3-2\alpha}{1-\alpha}$.

Indeed, the inequality in passing to the last line of (\ref{eq-R1term}) follows by writing
\begin{align*}
\frac{\log^a(\kappa+(4(\xi^2+|h_0|^2)^{1/2})^{-1})}{(\xi^2+|h_0|^2)^{(1/2)-\alpha}\xi_0^{2\alpha}}=\frac{\log^{\beta}(\kappa+(4\xi_0)^{-1})}{|h_0|\log^{a\frac{2-\alpha}{1-\alpha}}(\kappa+(4|h_0|)^{-1})}\cdot \Xi
\end{align*}
with
\begin{align*}
\Xi&:=\frac{|h_0|\log^{a\frac{2-\alpha}{1-\alpha}}(\kappa+(4|h_0|)^{-1})\log^a(\kappa+(4(\xi^2+|h_0|^2)^{1/2})^{-1})}{(\xi^2+|h_0|^2)^{(1/2)-\alpha}\xi_0^{2\alpha}\log^{\beta}(\kappa+(4\xi_0)^{-1})}
\end{align*}
and observing that an application of Lemma $\ref{lem-alm-mon}$, along with $\xi(t)\leq \xi_0$ and $|h_0|\leq \xi_0$, implies
\begin{align*}
\Xi&\leq C_2\frac{(\xi^2+|h_0|^2)^{\alpha}\log^{\beta}(\kappa+(4(\xi^2+|h_0|^2)^{1/2})^{-1})}{\xi_0^{2\alpha}\log^{\beta}(\kappa+(4\xi_0)^{-1})}\lesssim \frac{1}{\alpha^\beta}.
\end{align*}

Similarly, with the above choice of $R_1$, we also have
\begin{align*}
&\frac{\lVert v\rVert_{L^\infty}\xi^\alpha}{R_1}+\frac{\lVert v\rVert_{L^\infty}}{R_1^{1-\alpha}}\\
&\hspace{0.2in}\leq \frac{CM_{\theta,f}}{\xi_0^\alpha}\bigg(\frac{\xi^\alpha}{(\xi^2+|h_0|^2)^{1/2}}+\frac{1}{(\xi^2+|h_0|^2)^{\frac{1-\alpha}{2}}}\bigg)\\
&\hspace{0.2in}\leq \frac{2CM_{\theta,f}}{\xi_0^\alpha}\bigg(\frac{(\xi^2+|h_0|^2)^{\alpha/2}}{(\xi^2+|h_0|^2)^{1/2}}\bigg)\\
&\hspace{0.2in}\leq \frac{CM_{\theta,f}\alpha^{-a\frac{2-\alpha}{1-\alpha}}}{|h_0|\log^{a\frac{2-\alpha}{1-\alpha}}(\kappa+(4|h_0|)^{-1})}\log^{a\frac{2-\alpha}{1-\alpha}}(\kappa+(4\xi_0)^{-1}),
\end{align*}
where in the inequality passing from the third to fourth lines, we used
\begin{align*}
\frac{2CM_{\theta,f}}{\xi_0^\alpha}\bigg(\frac{(\xi^2+|h_0|^2)^{\alpha/2}}{(\xi^2+|h_0|^2)^{1/2}}\bigg)&=\frac{2CM_{\theta,f}\log^{a\frac{2-\alpha}{1-\alpha}}(\kappa+(4\xi_0)^{-1})}{|h_0|\log^{a\frac{2-\alpha}{1-\alpha}}(\kappa+(4|h_0|)^{-1})}\cdot\widetilde{\Xi}
\end{align*}
with
\begin{align*}
\widetilde{\Xi}&:=\frac{|h_0|\log^{a\frac{2-\alpha}{1-\alpha}}(\kappa+(4|h_0|)^{-1})}{\xi_0^\alpha\log^{a\frac{2-\alpha}{1-\alpha}}(\kappa+(4\xi_0)^{-1})}\cdot\frac{(\xi^2+|h_0|^2)^{\alpha/2}}{(\xi^2+|h_0|^2)^{1/2}}\\
&\leq C_2\frac{(\xi^2+|h_0|^2)^{1/2}\log^{a\frac{2-\alpha}{1-\alpha}}(\kappa+(4(\xi^2+|h_0|^2)^{1/2})^{-1})}{\xi_0^\alpha\log^{a\frac{2-\alpha}{1-\alpha}}(\kappa+(4\xi_0)^{-1})}\cdot\frac{(\xi^2+|h_0|^2)^{\alpha/2}}{(\xi^2+|h_0|^2)^{1/2}}\\
&=C_2\frac{(\xi^2+|h_0|^2)^{\alpha/2}\log^{a\frac{2-\alpha}{1-\alpha}}(\kappa+(4(\xi^2+|h_0|^2)^{1/2})^{-1})}{\xi_0^\alpha\log^{a\frac{2-\alpha}{1-\alpha}}(\kappa+(4\xi_0)^{-1})}\\
&\lesssim \frac{1}{\alpha^{a\frac{2-\alpha}{1-\alpha}}}.
\end{align*}

Turning to the last term on the right-hand side of (\ref{eq-gprime-1}), we make the observation that one has the bound $|F(t,x_0;h_0)|\leq M_{\theta,f}|h|^{-\alpha}$, and use the inequality $2ab\leq c_0a^2+c_0^{-1}b^2$, giving
\begin{align*}
2vF
&\leq \frac{c_0}{|h_0|\log^{a\frac{2-\alpha}{1-\alpha}}(\kappa+(4|h_0|)^{-1})}|v|^2\\
&\hspace{0.4in}+\frac{M_{\theta,f}^2}{c_0}|h_0|^{1-2\alpha}\log^{a\frac{2-\alpha}{1-\alpha}}(\kappa+(4|h_0|)^{-1}).
\end{align*}

Collecting the above estimates, it follows that we have the bound
\begin{align*}
&g'(t)+\frac{1}{2}|D_{h_0}[v(t)]|\\
&\hspace{0.2in}\leq (C_2+1)d_{h_0}[v]+\frac{C_3\alpha|h_0|^2}{\xi^2+|h_0|^2}\bigg(M_{\theta,f}^2\alpha^{1-\beta}\log^\beta(\kappa+(4\xi_0)^{-1})\\
&\hspace{1.0in}+M_{\theta,f}\alpha^{-a\frac{2-\alpha}{1-\alpha}}\log^{a\frac{2-\alpha}{1-\alpha}}(\kappa+(4\xi_0)^{-1})\bigg)\frac{|v|^2}{|h_0|\log^{a\frac{2-\alpha}{1-\alpha}}(\kappa+(4|h_0|)^{-1})}\\
&\hspace{1.8in}+\frac{M_{\theta,f}^2}{c_0}|h_0|^{1-2\alpha}\log^{a\frac{2-\alpha}{1-\alpha}}(\kappa+(4|h_0|)^{-1}),
\end{align*}
with constant $C_3$ depending only on $a$.

Now, $a<\frac{1}{2}$ implies the limits
\begin{align*}
\alpha^{2-\beta}=\alpha^{2-a\frac{3-2\alpha}{1-\alpha}}\rightarrow 0,\quad \alpha^{1-a\frac{2-\alpha}{1-\alpha}}\rightarrow 0
\end{align*}
hold as $\alpha\rightarrow 0$.  It follows that we may choose $\alpha_0=\alpha_0(a,M_{\theta,f},\xi_0)<1/2$ small enough so that for $\alpha\leq \alpha_0$ we have
\begin{align*}
\nonumber &C_3\alpha\bigg(M_{\theta,f}^2\alpha^{1-\beta}\log^\beta(\kappa+(4\xi_0)^{-1})\\
&\hspace{0.8in}+M_{\theta,f}\alpha^{-a\frac{2-\alpha}{1-\alpha}}\log^{a\frac{2-\alpha}{1-\alpha}}(\kappa+(4\xi_0)^{-1})\bigg)\leq c_0.
\end{align*}
Using this bound, we obtain
\begin{align*}
g'(t)+\frac{1}{2}|D_{h_0}[v(t)]|&\leq (C_2+1)d_{h_0}[v]+\frac{c_0}{|h_0|\log^{a\frac{2-\alpha}{1-\alpha}}(\kappa+(4|h_0|)^{-1})}|v|^2\\
&\hspace{0.2in}+\frac{M_{\theta,f}^2}{c_0}|h_0|^{1-2\alpha}\log^{a\frac{2-\alpha}{1-\alpha}}(\kappa+(4|h_0|)^{-1})\\
&=(C_2+2)d_{h_0}[v(t)]+\frac{M_{\theta,f}^2}{c_0}|h_0|^{1-2\alpha}\log^{a\frac{2-\alpha}{1-\alpha}}(\kappa+(4|h_0|)^{-1}).
\end{align*}
so that, using Lemma $\ref{lem-alm-mon}$ along with the bound $|h_0|\leq \xi_0$, and allowing the constant $C_4=C_4(a)$ to increase from line to line,
\begin{align}
\nonumber g'(t)+\frac{1}{2}|D_{h_0}[v(t)]|-(C_2+2)d_{h_0}[v]&\leq \frac{C_4M_{\theta,f}^2}{c_0}(\xi_0)^{1-2\alpha}\log^{a\frac{2-\alpha}{1-\alpha}}(\kappa+(4\xi_0)^{-1})\\
\nonumber &\leq \frac{C_4M_{\theta,f}^2}{c_0}(\xi_0)^{1-2\alpha}\log^{3a}(\kappa+(4\xi_0)^{-1})\\
\nonumber &\leq \frac{C_4M_{\theta,f}^2}{c_0}(\xi_0)^{(1/2)-2\alpha}\\
\nonumber &\leq \frac{C_4M^2}{c_0}\xi_0^{1/2}.
\end{align}

We now invoke the lower bound (\ref{eq-lower-bd-3}) for $D_{h_0}[v(t)]$, which in our present notation can be written as
\begin{align*}
(D_h[v(t)])(x)&\geq \frac{c_1}{c_0}\bigg(\frac{ |v(t,x;h)| }{ \lVert v\rVert_{L^\infty} }\bigg)^{1/(1-\alpha)}d_h[v(t)](x)
-\frac{C_5\|v\|_{L^\infty}^2}{r_0\log^a(\kappa+r_0^{-1})}.
\end{align*}
In view of our choice of $(x_0,h_0)$ as a point of maximum for $|v(t)|$, this becomes

\begin{align*}
(D_{h_0}[v(t)])(x_0)&\geq \bigg(\frac{c_1}{c_0}-\frac{C_5|h_0|\log^{a\frac{2-\alpha}{1-\alpha}}(\kappa+(4|h_0|)^{-1})}{c_0r_0\log^a(\kappa+r_0^{-1})}\bigg)d_{h_0}[v(t)](x_0)\\
&\geq \bigg(\frac{c_1}{c_0}-\frac{C_5\xi_0\log^{a\frac{2-\alpha}{1-\alpha}}(\kappa+(4\xi_0)^{-1})}{c_0r_0\log^a(\kappa+r_0^{-1})}\bigg)d_{h_0}[v(t)](x_0)\\
&\geq \bigg(\frac{c_1}{c_0}-\frac{C_5\xi_0^{1/2}}{c_0r_0\log^a(\kappa+r_0^{-1})}\bigg)d_{h_0}[v(t)](x_0),
\end{align*}
where the constant $C_5>0$ again depends only on $a$ and may increase from line to line.
Combining this with the above estimates gives
\begin{align*}
g'(t)+\bigg(\frac{c_1}{2c_0}-\frac{C_5\xi_0^{1/2}}{2c_0r_0\log^a(\kappa+r_0^{-1})}-(C_2+2)\bigg)d_{h_0}[v(t)]& \leq \frac{C_4M^2}{c_0}\xi_0^{1/2}.
\end{align*}
Recall that $c_1$ was determined at the end of Section 1.1 above, as the constant involved in the lower bound ($\ref{eq-lower-bd-3}$).  Choose $c_0=c_1/(4(C_2+2))$.  This gives
\begin{align*}
g'(t)+\bigg(C_2+2-\frac{2C_5(C_2+2)\xi_0^{1/2}}{c_1r_0\log^a(\kappa+r_0^{-1})}\bigg)d_{h_0}[v(t)]& \leq \frac{4C_4(C_2+2)M^2}{c_1}\xi_0^{1/2},
\end{align*}
so that for $\xi_0\leq X_1(a,r_0)$ with
\begin{align*}
X_1(a,r_0):=\bigg(\frac{c_1r_0\log^a(\kappa+r_0^{-1})(C_2+2)}
{4C_5}\bigg)^2,
\end{align*}
we have
\begin{align}
g'(t)+\frac{C_2+2}{2}d_{h_0}[v(t)]& \leq \frac{4C_4(C_2+2)M^2}{c_1}\xi_0^{1/2}.\label{eq-g-2}
\end{align}

Observing that
\begin{align*}
d_{h_0}[v(t)]\geq \frac{c_0}{4\pi\log^{a\frac{2-\alpha}{1-\alpha}}(\kappa+(16\pi)^{-1})}v^2&\geq \frac{c_0}{4\pi\log^{3a}(\kappa+(16\pi)^{-1})}v^2\\
&=\frac{c_1}{4A(C_2+2)}v^2,
\end{align*}
with $A:=4\pi\log^{3a}(\kappa+(16\pi)^{-1})$, we rewrite the bound ($\ref{eq-g-2}$) as
\begin{align}
g'(t)+\bigg(\frac{c_1}{8A}\bigg)g(t)&\leq \frac{4C_4(C_2+2)M^2}{c_1}\xi_0^{1/2}.\label{eq-g-3}
\end{align}
for $\xi_0\leq X_1(a,r_0)$.

We now make our choice of constants $\xi_0$ and $\alpha$.
For this, note that (\ref{eq-g-3}) and the observation that $g(t)\geq (3M/4)^2$ for $t_0\leq t\leq T$ imply
\begin{align*}
g'(t)\leq \frac{1}{A}\bigg(-\frac{9c_1}{128}+\frac{4C_4(C_2+2)}{c_1}\xi_0^{1/2}A\bigg)M^2,
\end{align*}
for $\xi_0\leq X_1(a,r_0)$.  Choosing
$$X_0=X_0(a,r_0)=\min\bigg\{X_1(a,r_0),\left(\frac{9c_1^{2}}{2^{10}A
C_4(C_2+2)}\right)^2\bigg\},$$
where the right side is a constant depending only on $a$ and $r_0$, we obtain that, for $\xi_0\leq X_0$,
\begin{align*}
\frac{4C_4(C_2+2)}{c_1}\xi_0^{1/2}A\leq \frac{9c_1}{256},
\end{align*}
and thus $g'(t)\leq -9c_1M^2/(256)<0$ for $t_0\leq t\leq T$.  We therefore have $g(t)\leq (3M/4)^2$ for $t\in (t_0,T)$, which gives the desired contradiction when evaluated at $t=T$.

We have thus shown that for $\xi_0\leq X_0(a)$ we may choose $\alpha_0(a,M_{\theta,f},\xi_0)$ such that under the hypotheses described in the statement of the theorem, $$\lVert v(t)\rVert_{L^\infty}\leq M$$ for all $t\geq 0$.  Recalling the definition of $g$ and $v$, we have that for $t\geq T_*$, $\xi(t)=0$, and thus
\begin{align*}
|\theta(t)|_{C^\alpha}&=\lVert v(t)\rVert_{L^\infty}\leq M,\quad 0<\alpha\leq \alpha_0,
\end{align*}
as desired.
\end{proof}

We are now ready to conclude the global wellposedness of the equation \eqref{eq1-log}.

\begin{proof}[Proof of Theorem \ref{thm-log-gwp}]
It follows from Proposition \ref{prop-local} that a unique local solution $u$ exists up to time $T>0$, which depends on the initial data. Now we choose a sufficiently small $\alpha>0$ so that $T_*\le T/2$, where $T_*$ is from Proposition \ref{prop4}. By the a priori $C^\alpha$ estimate in Proposition \ref{prop4} and the known regularity criteria, we conclude that $u$ is global-in-time.

Next, note that the $C^\alpha$ norm is supercritical with respect to the scaling of (\ref{eq1-log}).  A standard bootstrap argument (see, for instance \cite[Theorem 3.1]{CW} and analogous arguments in \cite{CTV}) therefore gives
\begin{align*}
\sup_{t>T_*}\,\,\lVert \theta(t)\rVert_{H^k}<\infty
\end{align*}
as desired.
\end{proof}

\section{Proof of Theorem \ref{thm2}: $(H^\tau,L^2)$ nonlinear instability for forced log-supercritical SQG}

In this section, we establish Theorem $\ref{thm2}$, which asserts that if the log-supercritical SQG evolution is linearly-unstable near the stationary solution $\Theta_0$ then it is $(H^\tau,L^2)$ nonlinearly unstable with respect to perturbation of $\Theta_0$.

We first state and prove a version of Lemma $\ref{lemma-ldelta-supercritical}$ adapted to the logarithmically supercritical equation.

\begin{lemma}
\label{lemma-ldelta-log}
Fix $a\in [0,1)$, $\sigma\in [0,1]$, and let $\mathcal{L}$ be as stated earlier.  Suppose that $\Theta_0$ satisfies $\int \Theta_0=0$ (and thus $R^\perp\Theta_0$ has the same property, e.g. by observing that this property corresponds to vanishing of the zeroth Fourier coefficient).  Then there exists a constant $C>0$ such that for all $t\neq 0$ and $\varphi\in C^{\infty}(\mathbb{T}^2)$ with $\int_{\mathbb{T}^2} \varphi(x)dx=0$ one has
\begin{align*}
\Vert e^{tL_{\textrm{log},\delta}}\varphi\rVert_{L^2}\leq Ct^{-\sigma}\lVert \varphi\rVert_{L^2}^{1-\sigma}\lVert \mathcal{L}^{-1}\varphi\rVert_{L^2}^\sigma,
\end{align*}
where $L_{\textrm{log},\delta}=L_{\textrm{log},\delta}=L_{\textrm{log}}-(\lambda+\delta)I$ as in (\ref{def-ldelta}), adapted to the linearized operator $L_{\textrm{log}}$.
\end{lemma}

\begin{proof}
As in the proof of Lemma $\ref{lemma-ldelta-supercritical}$, the essential task is to establish a bound of the form
\begin{align}
\lVert \phi\rVert_{L^2}\lesssim \lVert \mathcal{L}^{-1}A\phi\rVert_{L^2}\label{eq2-log}
\end{align}
for $\phi$ satisfying $\int \phi dx=0$, where the operator $A$ is given by
$$A\varphi=L_{\textrm{log},\delta}\varphi+(R^\perp \varphi)\cdot\nabla\Theta_0-(\alpha-\lambda-\delta)\varphi.$$

Writing
\begin{align*}
\int (\mathcal{L}^{-1}A\phi)\overline{\phi}dx=-\int \Big[\mathcal{L}^{-1}((R^\perp\Theta_0)\cdot \nabla \phi)\Big]\overline{\phi}dx-\lVert \phi\rVert_{L^2}^2-\alpha\lVert \mathcal{L}^{-1/2}\phi\rVert_{L^2}^2,
\end{align*}
we have
\begin{align*}
\lVert \phi\rVert_{L^2}^2+\alpha\lVert \mathcal{L}^{-1/2}\phi\rVert_{L^2}^2&=-\int (\mathcal{L}^{-1}A\phi)\overline{\phi} dx-\int \Big[\mathcal{L}^{-1}((R^\perp\Theta_0)\cdot\nabla\phi)\Big]\overline{\phi} dx\\
&\leq \lVert \mathcal{L}^{-1}A\phi\rVert_{L^2}\lVert\phi\rVert_{L^2}+\bigg|\int \Big[(R^\perp\Theta_0)\cdot \nabla\phi\Big]\overline{\Big[\mathcal{L}^{-1}\phi\Big]}dx\bigg|,
\end{align*}
so that if
\begin{align*}
(R^\perp\Theta_0)(x)=\sum_{j\in\mathbb{Z}^2, j\neq 0} \widehat{R}_je^{2\pi i j\cdot x},\quad
\phi(x)=\sum_{k\in\mathbb{Z}^2, k\neq 0} \widehat{\phi}_ke^{2\pi i k\cdot x},
\end{align*}
an argument as in (\ref{eq-rperp})--(\ref{eq3}) gives
\begin{align}
\nonumber &\bigg|\int \Big[(R^\perp\Theta_0)\cdot\nabla\phi\Big]\overline{\Big[\mathcal{L}^{-1}\phi\Big]}dx\bigg|\\
\nonumber &\hspace{0.2in}\leq \sum_{\substack{j,k,\ell\neq 0\\j+k-\ell=0}} |\widehat{R}_j|\,|\widehat{\phi}_k|\, |\widehat{\phi}_\ell|\, \bigg|\frac{|k|\log^{a/2}(\kappa+|\ell|)}{|\ell|^{1/2}}\\
&\hspace{2.1in}\cdot\bigg(\frac{\log^{a/2}(\kappa+|\ell|)}{|\ell|^{1/2}}-\frac{\log^{a/2}(\kappa+|k|)}{|k|^{1/2}}\bigg)\bigg|.\label{eq3-log}
\end{align}

As before, the three vectors $j$, $k$, and $\ell$ form a triangle, with length of the two largest sides comparable.  We again have three cases.  Set $m=\min \{|j|,|k|,|\ell|\}$.  If $m=|\ell|$, then we have $|j|\sim |k|$, and thus
\begin{align*}
&\frac{|k|\log^{a/2}(\kappa+|\ell|)}{|\ell|^{1/2}}\bigg|\frac{\log^{a/2}(\kappa+|\ell|)}{|\ell|^{1/2}}-\frac{\log^{a/2}(\kappa+|k|)}{|k|^{1/2}}\bigg|\\
&\hspace{0.2in}\lesssim \frac{|k|\log^{a/2}(\kappa+|\ell|)}{|\ell|^{1/2}}
\sim \frac{|j|\log^{a/2}(\kappa+|\ell|)}{|\ell|^{1/2}}.
\end{align*}
Alternatively, if $m=|k|$, then $|\ell|\sim |j|$, and
\begin{align*}
&\frac{|k|\log^{a/2}(\kappa+|\ell|)}{|\ell|^{1/2}}\bigg|\frac{\log^{a/2}(\kappa+|\ell|)}{|\ell|^{1/2}}-\frac{\log^{a/2}(\kappa+|k|)}{|k|^{1/2}}\bigg|\\
&\hspace{0.2in}\lesssim \frac{|k|^{1/2}\log^{a/2}(\kappa+|\ell|)\log^{a/2}(\kappa+|k|)}{|\ell|^{1/2}}\\
&\hspace{0.2in}\lesssim \frac{|j|\log^{a/2}(\kappa+|\ell|)}{|\ell|^{1/2}},
\end{align*}
while if $m=|j|$, then $|k|\sim |\ell|$ and by the mean value theorem,
\begin{align*}
&\frac{|k|\log^{a/2}(\kappa+|\ell|)}{|\ell|^{1/2}}\bigg|\frac{\log^{a/2}(\kappa+|\ell|)}{|\ell|^{1/2}}-\frac{\log^{a/2}(\kappa+|k|)}{|k|^{1/2}}\bigg|\\
&\hspace{0.2in}\lesssim \frac{|k|\log^{a}(\kappa+|\ell|)|j|}{|\ell|^{2}}\\
&\hspace{0.2in}\lesssim \frac{|j|\log^{a/2}(\kappa+|\ell|)}{|\ell|^{1/2}}.
\end{align*}

Combining the above with the Cauchy-Schwarz inequality, we get, for any $\varepsilon>0$,
\begin{align*}
(\ref{eq3-log})&\leq C\sum_{j,\ell\neq 0} \frac{|\widehat{R}_j|\, |\widehat{\phi}_{\ell-j}|\, |\widehat{\phi}_\ell|\, |j|\, \log^{a/2}(\kappa+|\ell|)}{|\ell|^{1/2}}\\
&\leq C\lVert \Lambda^{2+\varepsilon}\Theta_0\rVert_{L^2}\lVert \phi\rVert_{L^2}\lVert \mathcal{L}^{-1/2}\phi\rVert_{L^2}\\
&\leq \frac{1}{2}\lVert \phi\rVert_{L^2}^2+C\lVert \Lambda^{2+\varepsilon}\Theta_0\rVert_{L^2}^2\lVert \mathcal{L}^{-1/2}\phi\rVert_{L^2}^2,
\end{align*}
and thus
\begin{align*}
\frac{1}{2}\lVert \phi\rVert_{L^2}^2+\alpha\lVert \mathcal{L}^{-1/2}\phi\rVert_{L^2}^2\leq \lVert \mathcal{L}^{-1}A\phi\rVert_{L^2}\lVert \phi\rVert_{L^2}+C\lVert \Lambda^{2+\varepsilon}\Theta_0\rVert_{L^2}^2\lVert \mathcal{L}^{-1/2}\phi\rVert_{L^2}^2
\end{align*}
so that if $\alpha$ is chosen sufficiently large depending on $C$ and $\lVert \Lambda^{2+\varepsilon}\Theta_0\rVert_{L^2}$, 
we obtain (\ref{eq2-log}) as desired.
\end{proof}

We now prove Theorem $\ref{thm2}$.  As in the proof of Theorem $\ref{thm1}$, we let $\varphi$ be an eigenfunction associated to the eigenvalue $\mu$, and we will show that there exists a constant $c_0>0$ such that for all $\varepsilon>0$, the log-supercritical evolution of $\tth$ with data $\tth(0,\cdot)=\varepsilon \varphi$ will satisfy $$\lVert \tth(t_*)\rVert_{L^2}\geq c_0$$ for some $t_*>0$.

\begin{proof}[Proof of Theorem \ref{thm2}]
Fix $\varepsilon>0$, and let $\tth$ denote the solution to (\ref{eq-perturbation}) with initial data $\tth(0)=\varepsilon\varphi$. Thanks to the global wellposedness result for (\ref{eq1-log}) given by Theorem \ref{thm-log-gwp}, this solution is global in time, with
\begin{align*}
\sup_{t>t_0}\,\, \lVert \tth(t)\rVert_{H^3}\leq \sup_{t>t_0}\,\, \|\theta(t)\|_{H^3}+\|\Theta_0\|_{H^3}\leq C(a,k,r_0,t_0,\Theta_0,f),
\end{align*}
for all $t_0>0$, where $\theta=\Theta_0+\tth$ solves (\ref{eq1-log}) and where the constant $C(a,k,r_0,t_0,\Theta_0,f)$ may be chosen to be independent of $\varepsilon$.  Moreover, since $f\in H^3(\bT^2)$, by using a bootstrap argument which is similar to the proof of Proposition \ref{prop2}, we have $\Theta_0\in H^{7/2}(\bT^2)$ and $\varphi\in H^{5/2}(\bT^2)$. Therefore, the above estimate together with the local theory gives
$$
\sup_{t>0}\,\, \lVert \tth(t)\rVert_{H^{5/2}}\leq  C,
$$
where $C$ is independent of $\varepsilon$.

As in the proof of Theorem $\ref{thm1}$, we remark that the Duhamel formula gives
\begin{align*}
\tth(t)
&=\varepsilon e^{tL}\varphi+\int_0^t e^{(\lambda+\delta)(t-s)}e^{(t-s)L_{\textrm{log},\delta}}N(\tth)(s)\, ds
\end{align*}
for all $t\in [0,+\infty)$.

Let $t\in [0,\infty)$ be given and fix parameters $\sigma\in (0,1)$ and $A>\lVert \varphi\rVert_{L^2}$ to be specified later.  Define $0<T\leq +\infty$ by setting $$T:=\sup\bigg\{\tau>0:\lVert \tth(t)\rVert_{L^2}\leq A\varepsilon e^{\lambda t}\,\, \textrm{for all}\,\, 0\leq t\leq \tau\bigg\},$$ and note that the set inside the supremum is nonempty as a consequence of the continuity of $t\mapsto \tth(t)$ in $L^2$ and the lower bound on the choice of the parameter $A$.

The Minkowski inequality then gives
\begin{align}
\nonumber \varepsilon\lVert e^{tL}\varphi\rVert_{L^2}&\leq \lVert \tth(t)\rVert_{L^2}+\int_0^t e^{(\lambda+\delta)(t-s)}\lVert e^{(t-s)L_{\textrm{log},\delta}}N(\tth)(s)\rVert_{L^2}\, ds,
\end{align}
so that by Lemma $\ref{lemma-ldelta-log}$, and the definition (\ref{def-n}) of $N(\tth)$, we have
\begin{align*}
\lVert \tth(t)\rVert_{L^2}&\geq \varepsilon\lVert e^{tL}\varphi\rVert_{L^2}-CB(t;\tth)\\
&=\varepsilon e^{\lambda t}\lVert \varphi\rVert_{L^2}-CB(t;\tth),\quad \textrm{for}\quad 0\leq t\leq T,
\end{align*}
where we set
\begin{align}
B(t;\tth):=\int_0^t \frac{e^{(\lambda+\delta)(t-s)}}{(t-s)^{\sigma}}\lVert \tu(s)\cdot \nabla\tth(s)\rVert_{L^2}^{1-\sigma}\lVert \mathcal{L}^{-1}[\tu(s)\cdot\nabla\tth(s)]\rVert_{L^2}^\sigma\, ds\label{def-B-log}
\end{align}
and
\begin{align*}
\tu:=R^\perp\tth.
\end{align*}

Fix $t\in [0,T]$, $s\in (0,t)$ and write $\tu=\tu(s)$, $\nabla\tth=\nabla\tth(s)$.  Then
\begin{align*}
\|\tu\cdot\nabla\tth\|_{L^2}&\le \|\tu\|_{L^{2}}\|\nabla\tth\|_{L^{\infty}}\\
&\le C\|\tth\|_{L^{2}}\|\tth\|_{\dot{H}^{5/2}},
\end{align*}
where we used the boundedness of Riesz transforms on $L^2$ and the two-dimensional Sobolev embedding $H^{3/2}(\mathbb{T}^2)\hookrightarrow L^\infty(\mathbb{T}^2)$.

Similarly, by the fractional Leibniz rule, the boundedness of Riesz transforms in $L^p$, we get
\begin{align*}
\|\mathcal{L}^{-1}(\tu\cdot\nabla\tth)\|_{L^2}&=\|\mathcal{L}^{-1}\divop (\tu \tth)\|_{L^2}\\
&\le C\|\Lambda^\rho (\tu \tth)\|_{L^2}\\
&\le C\|\tth\|_{L^4}\|\Lambda^\rho \tth\|_{L^4}\\
&\le C\|\tth\|_{L^2}^{2(4-\rho)/5}\|\tth\|_{\dot{H}^{5/2}}^{2(1+\rho)/5}.
\end{align*}
where we used the Gagliardo-Nirenberg inequalities
\begin{align*}
\|\tth\|_{L^4}&\lesssim \|\tth\|_{\dot{H}^{1/2}}\lesssim \|\tth\|_{L^2}^{4/5}\|\tth\|_{\dot{H}^{5/2}}^{1/5}
\end{align*}
and
\begin{align*}
\|\Lambda^\rho\tth\|_{L^4}&\lesssim \|\tth\|_{\dot{H}^{(1/2)+\rho}}\lesssim \|\tth\|_{L^2}^{(4-2\rho)/5}\|\tth\|_{\dot{H}^{3}}^{(1+2\rho)/5}.
\end{align*}

We now choose $\rho$ sufficiently close to $0$ and $\sigma$ sufficiently close to $1$ to get
$$
\|\tu\cdot \nabla\tth\|_{L^2}^{1-\sigma}\|\mathcal{L}^{-1}(\tu\cdot \nabla\tth)\|_{L^2}^{\sigma}
\le C \|\tth\|_{L^2}^{1+\beta} \|\tth\|_{\dot{H}^{5/2}}^{1-\beta}
$$
for some $\beta\in (0,1)$ (for instance, choosing $\rho=1/4$ and $\sigma=3/4$ leads to $\beta=3/8$).

Combining this bound with the definition (\ref{def-B-log}) of $B(t;\tth)$, the rest of the argument proceeds as in the proof of Theorem $\ref{thm1}$ above.  We sketch the argument for the convenience of the reader.  In particular, if we require $0<\delta<\lambda\beta/2$ on $\delta$, and recall that our choice of $T$ gives $\|\tth(s)\|\leq A\varepsilon e^{\lambda s}$ for $s\leq t$, then we get
\begin{align*}
B(t;\tth)&\leq C_1 \int_0^t \frac{e^{\lambda(1+\beta)(t-s)}}{(t-s)^{\sigma}}
\|\tth(s)\|_{L^2}^{1+\beta}e^{-\frac{\lambda\beta}{2}(t-s)}\, ds
\leq C_1(A\varepsilon e^{\lambda t})^{1+\beta}
\end{align*}
for $0\leq t\leq T$.

Then, choosing both
\begin{align*}
C_2:=\left[\frac{A-\lVert \varphi\rVert_{L^2}}{C_1A^{1+\beta}}\right]^{1/\beta}
\end{align*}
and $A=3\lVert \varphi\rVert_{L^2}/2$ as in the proof of Theorem $\ref{thm1}$ gives
\begin{align*}
\lVert \tth(t_*)\rVert_{L^2}\geq \left[\frac{1}{3^{\beta+1}C_1}\right]^{1/\beta}
\end{align*}
with $0<t_*=\lambda^{-1}\log(C_2/\varepsilon)<T$.  This completes the argument.
\end{proof}


\begin{thebibliography}{99}

\bibitem{CC} A C\'ordoba, D. C\'ordoba.  A maximum principle applied to quasi-geostrophic equations. \textit{Comm. Math. Phys.} 249 (2004), 511--528.

\bibitem{CTV} P. Constantin, A. Tarfulea and V. Vicol.  Long time dynamics of forced critical SQG.  \textit{Comm. Math. Phys.} 335 (2015), no. 1, 93--141.

\bibitem{CV} P. Constantin and V. Vicol.  Nonlinear maximum principles for dissipative linear nonlocal operators and applications.  \textit{Geom. Funct. Anal.} \textbf{22} (2012), no. 5, 1289--1321.

\bibitem{CZV} M. Coti Zelati and V. Vicol.  On the global regularity for the supercritical SQG equation.  \textit{Indiana Univ. Math. J.} 65 (2016), 535--552.

\bibitem{CW} P. Constantin and J. Wu.  Regularity of H\"older continuous solutions of the supercritical quasi-geostrophic eqaution, \textit{Ann. Inst. H. Poincar\'e Analyse Non Lin\'eaire} \textbf{25} (2008), 1103--1110.

\bibitem{Dab} M. Dabkowski.  Eventual regularity of the solutions to the supercritical dissipative quasi-geostrophic equation. \textit{Geom. Funct. Anal.} 21 (2011), no. 1, 1--13.

\bibitem{DKSV} M. Dabkowski, A. Kiselev, L. Silvestre and V. Vicol.  Global well-posedness of slightly supercritical active scalar equations, \textit{Anal. PDE} \textbf{7} (2014), no. 1, 43--72.

\bibitem{D08} H. Dong, Well-posedness for a transport equation with nonlocal velocity, \textit{J. Funct. Anal.} \textbf{255} (2008), no. 11, 3070--3097.

\bibitem{FPV} S. Friedlander, N. Pavlovic, V. Vicol.  Nonlinear instability for the critically dissipative quasi-geostrophic equation, \textit{Comm. Math. Phys.} \textbf{292} (2009), no. 3, 797--810.

\bibitem{FSV} S. Friedlander, W. Strauss and M. Vishik.  Nonlinear instability in an ideal fluid, \textit{Ann. Inst. H. Poincar\'e Anal. Non Lin\'eaire.}  \textbf{14} (1997), no. 2, 187--209.

\bibitem{J06} N. Ju.  Dissipative quasi-geostrophic equation: local well-posedness, global regularity and similarity solutions.  \textit{Indiana Univ. Math. J.} 56 (2007), no. 1, 187--206.

\bibitem{MX} C. Miao and L. Xue.  On the regularity issues of a class of drift-diffusion equations with nonlocal diffusion.  \textit{SIAM J. Math. Anal.} 51 (2019), no. 4, 2927--2970.

\bibitem{M06} H. Miura, Dissipative quasi-geostrophic equation for large initial data in the critical Sobolev space, \textit{Comm. Math. Phys.} \textbf{267} (2006), no. 1, 141--157.

\bibitem{SGS} H. Shang, Y. Guo and M. Song.  Global regularity for the supercritical active scalars.  \textit{Z. Angew. Math. Phys.} (2017) 68:64.

\bibitem{Sil} L. Silvestre.  Eventual regularization for the slightly supercritical quasi-geostrophic equation. \textit{Ann. Inst. H. Poincaré Anal. Non Linéaire} 27 (2010), no. 2, 693--704.

\bibitem{XY} L. Xue and Z. Ye.  On the Differentiability issue of the drift-diffusion equation with nonlocal Lévy-type diffusion. \textit{Pacific J. Math.} 293 (2018) 471--510.

\bibitem{XZ} L. Xue and X. Zheng.  Note on the well-posedness of a slightly supercritical surface quasi-geostrophic equation.  \textit{J. Diff. Eq.} 253 (2012), no. 2, 795--813.

\bibitem{Yu} V. Yudovich.  \textit{The Linearization Method in Hydrodynamical Stability Theory.}  Translations of Math. Monographs, \textbf{74}, 1989, AMS.
\end{thebibliography}
\end{document}